\documentclass{article}

\usepackage{graphicx}
\usepackage{amsmath}
\usepackage{amsthm}
\usepackage{amssymb}
\usepackage{amsbsy}
\usepackage{amsfonts}
\usepackage{amstext}
\usepackage{amscd}
\usepackage[dvips]{epsfig}
\allowdisplaybreaks

\ifx\undefined\bysame
\newcommand{\bysame}{\leavevmode\hbox to3em{\hrulefill}\,}
\fi

\theoremstyle{plain}
    \newtheorem{thm}{Theorem}[section]
    \newtheorem{lem}[thm]{Lemma}
    \newtheorem{prop}[thm]{Proposition}
    \newtheorem{cor}[thm]{Corollary}
 \theoremstyle{definition}

    \newtheorem{rem}[thm]{Remark}
    \newtheorem{prob}[thm]{Problem}

\newcommand{\df}{\displaystyle\frac}

\title{\bf{On the $n$-loop Kontsevich invariant of knots having the same Alexander polynomial}}
\author{Kouki YAMAGUCHI}
\date{}

\begin{document}\large

\maketitle

\begin{abstract}
The $n$-loop Kontsevich invariant of knots takes its value in the completion of the space of $n$-loop open Jacobi diagrams, which is an infinite dimensional vector space. Since the 1-loop part is presented by the Alexander polynomial, we are interested in the image of the $\geq 2$-loop Kontsevich invariant of knots having the same Alexander polynomial. In this paper, we show that for $n\geq 2$ the subspace generated by the image of the $n$-loop Kontsevich invariant of genus $\leq g$ knots having the same Alexander polynomial is finite dimensional. Further, we give some concrete calculations about those subspaces and dimensions in several simple cases. 
\end{abstract}

\section{Introduction}
The Kontsevich invariant of knots is a powerful invariant, which dominates all quantum invariants derived from simple Lie algebras and all Vassiliev invariants of knots. The Kontsevich invariant takes its value in the completion of the space of Jacobi diagrams, which is an infinite dimensional vector space. Here, Jacobi diagrams are some kind of uni-trivalent graphs. Thus, the value of the Kontsevich invariant is presented as an infinite sum of Jacobi diagrams, and it is so hard to determine all terms of it. The $n$-loop part of the Kontsevich invariant $Z^{(n)}$, which we call the $n$-loop Kontsevich invariant, takes its value in the completion of the space of $n$-loop open Jacobi diagrams, which is also an infinite dimensional vector space. Here, for an open Jacobi diagram, we call it a $n$-loop diagram if its first Betti number is $n$. Since it is well-known that the 1-loop Kontsevich invariant is equivalent to the Alexander polynomial, we are interested in the image of the $\geq 2$-loop Kontsevich invariant of knots having the same Alexander polynomial. However, even though we fix a loop number $n\geq 2$, the value of the $n$-loop Kontsevich invariant is presented as an infinite sum of $n$-loop Jacobi diagrams, and it is not so easy to determine the value of it. Let $\Delta(t)\in\mathbb{Z}[t^{\pm 1}]$ be a polynomial which satisfies $\Delta(1)=1, \Delta(t)=\Delta(t^{-1})$, and let $\mathcal{K}^{\Delta(t)}$ be a set of all knots whose values of Alexander polynomial are $\Delta(t)$. Then, we can show that the subspace $\mathcal{V}(n,\Delta(t))$, which is the subspace generated by $Z^{(n)}(\mathcal{K}^{\Delta(t)})$, is infinite dimensional (see Appendix). This means that $Z^{(n)}(\mathcal{K}^{\Delta(t)})$ have the information of infinitely many Vassiliev invariants of those knots.\par
Then, we consider the filtration by the genus of knots. Let $\mathcal{K}_{\leq g}^{\Delta(t)}$ be a set of all genus $\leq g$ knots whose values of Alexander polynomial are $\Delta(t)$, noting that this also consists of infinitely many knots unless it is $\emptyset$ (see Appendix). Let $\mathcal{V}(n,g,\Delta(t))$ the subspace generated by $Z^{(n)}(\mathcal{K}_{\leq g}^{\Delta(t)})$. Then, we have the filtration of those subspaces, $\mathcal{V}(n,1,\Delta(t))\subset \mathcal{V}(n,2,\Delta(t))\subset\cdots\subset \mathcal{V}(n,\infty,\Delta(t)):=\displaystyle\bigcup_{g\geq 1}\mathcal{V}(n,g,\Delta(t))=\mathcal{V}(n,\Delta(t))$. In this paper, we show that for $n\geq 2$ the subspace $\mathcal{V}(n,g,\Delta(t))$ is finite dimensional. Let $d(n,g,\Delta(t))$ be its dimension. This finiteness means that $Z^{(n)}(\mathcal{K}_{\leq g}^{\Delta(t)})$ have the information of only $d(n,g,\Delta(t))$ numbers of Vassiliev invariants of $\mathcal{K}_{\leq g}^{\Delta(t)}$ and  we can present those values of $n$-loop Kontsevich invariant by using those Vassiliev invariants. Thus, roughly speaking, $d(n,g,\Delta(t))$ represents ``how many Vassiliev invariants determine the image of the $n$-loop Kontsevich invariant of $\mathcal{K}_{\leq g}^{\Delta(t)}$''. In \cite{Oh4}, it is shown that the $t_1$-degree of the 2-loop polynomial, which is the polynomial presenting $Z^{(2)}$, is bounded by the twice of the genus of knots. Though there is no single polynomial which can present $Z^{(n)}$ for large $n$, the fact that $\mathcal{V}(n,g,\Delta(t))$ is finite dimensional can be regarded as one of its generalization.\par
For its proof, we use the method of the computation of the loop expansion of the Kontsevich invariant of knots \cite{Ga1, Kri1}, which states that the $n$-loop Kontsevich invariant can be presented by finitely many polynomials. In order to compute the loop expansion, we use the specific method related to the calculation of LMO invariant, called the rational version of the Aarhus integral  \cite{BGRT1,BGRT2,BGRT3,Ga1,Kri1}. \par
Further, we give some concrete calculations about those dimensions in several simple cases. We determine $\mathcal{V}(2,1,\Delta(t))$ and $d(2,1,\Delta(t))$ for any $\Delta(t)$ completely, and we give another formula for the 2-loop Kontsevich invariant of genus 1 knots. Also, we give some upper bounds for $d(2,g,\Delta(t))$ and $d(3,1,1)$. Here is a table of these results.
\begin{align*}
\begin{tabular}{lrr}\hline
$\mathcal{V}(2,1,\Delta(t))$&$\mathcal{V}(2,g,\Delta(t))$&$\mathcal{V}(3,1,1)$\\ \hline
$d(2,1,\Delta(t))=2$\quad(unless it is 0)&$d(2,g,\Delta(t))\leq g^2+2g$&$d(3,1,1)\leq 11$\\ 
We have a basis&We have a generating set&We have a generating set\\ \hline
\end{tabular}
\end{align*}\par
This paper is organized as follows. In Section \ref{sec2}, we review the Kontsevich invariant of knots. In Section \ref{sec3}, we state the main results. In Section \ref{sec4}, we briefly review the rational version of the Aarhus integral and clasper surgeries. In Section \ref{sec5}, we prove the main results. In Appendix, we prove some nontrivial facts stated in Section \ref{sec3}.\par
This work was supported by JSPS KAKENHI Grant Number JP24KJ1326. The author would like to thank Tomotada Ohtsuki and Katsumi Ishikawa for helpful discussions and comments.

\section{Review of the Kontsevich invariant of knots}
\label{sec2}
In this section, we review the Kontsevich invariant of knots. In this paper, we assume that all knots are oriented and 0-framed unless otherwise stated, and the values of Alexander polynomial of knots belong to $\mathcal{Z}$, where $\mathcal{Z}$ is defined as follows,
\begin{align*}
\mathcal{Z}=\{f(t)\in\mathbb{Z}[t^{\pm 1}]\mid f(1)=1, f(t)=f(t^{-1})\}.
\end{align*}
Any element $f(t)\in\mathcal{Z}$ can be written as $f(t)=1+\sum_{j=1}^{m}a_j(t+t^{-1}-2)^j$ for $a_j\in\mathbb{Z}$, $a_m\neq 0$, and we denote $\text{deg}f:=m$.\par
For an oriented compact 1-manifold $X$, a {\it Jacobi diagram} on $X$ is an uni-trivalent graph together with $X$ such that univalent vertices are distinct points on $X$ and each trivalent vertex is vertex-oriented. Here, we say a trivalent vertex is ``{\it vertex-oriented }'' when a cyclic order of the three edges around the trivalent vertex is fixed. In figures, we draw $X$ by thick lines, the uni-trivalent graphs by thin lines, and, each trivalent vertex is vertex-oriented by counterclockwise order. We sometimes call $X$ a {\it skeleton} of a Jacobi diagram. We define $\mathcal{A}(X)$ to be the quotient $\mathbb{Q}$-vector space spanned by Jacobi diagrams on $X$ subject to the following relations, called the AS, IHX, and STU relations,
\begin{align*}
\mbox{$\begin{array}{c}
   \includegraphics[scale=0.3]{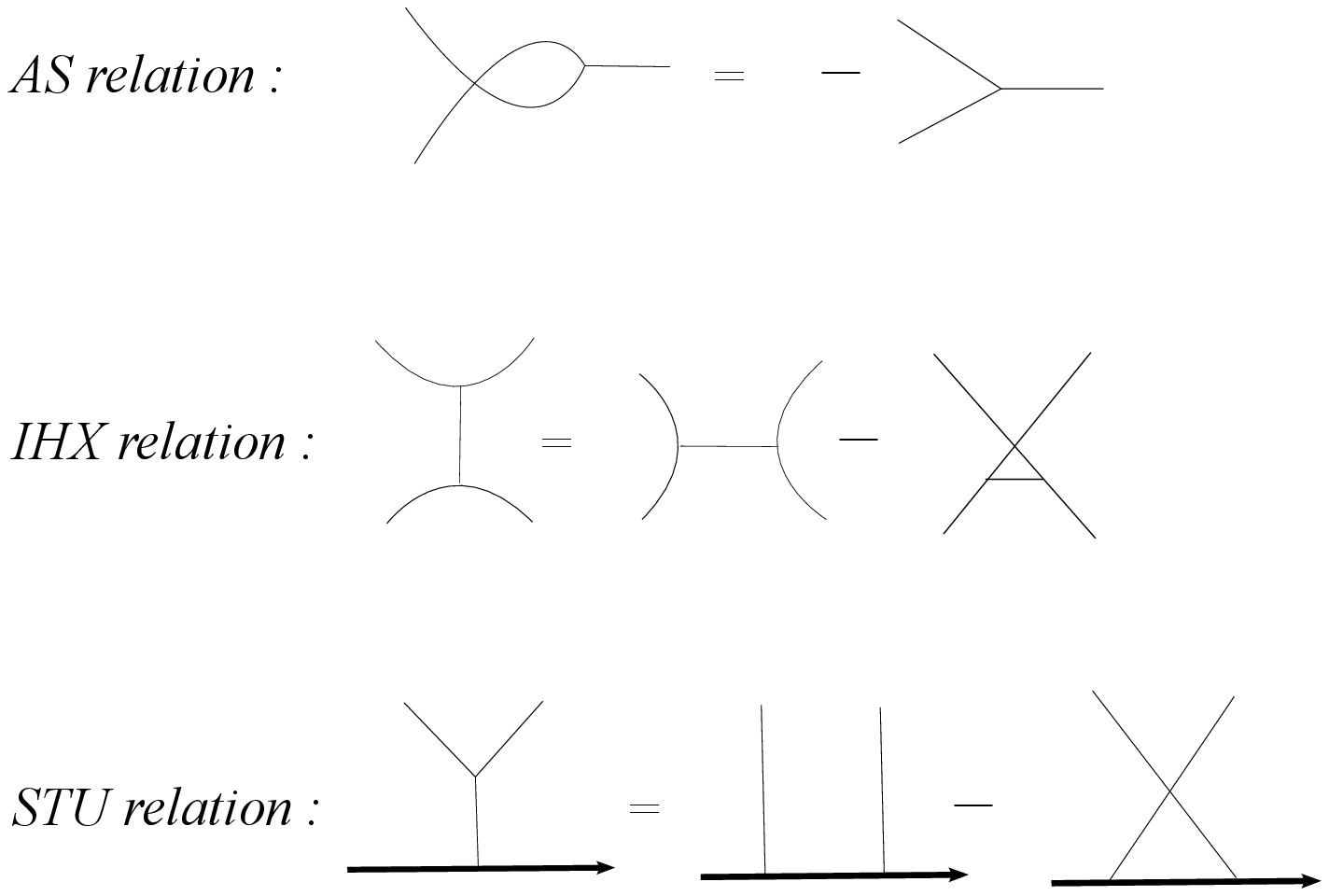}
   \end{array}$}.
\end{align*}
An {\it open Jacobi diagram} is a vertex-oriented uni-trivalent graph. We define $\mathcal{B}$ to be the quotient $\mathbb{Q}$-vector space spanned by open Jacobi diagrams subject to the AS, IHX relations. We sometimes write $\mathcal{B}$ as $\mathcal{A}(\ast)$. For an open Jacobi diagram, we call it a {\it $n$-loop diagram} if its first Betti number is $n$. We define $\mathcal{B}^{(n)}$ to be the quotient $\mathbb{Q}$-vector space spanned by $n$-loop open Jacobi diagrams subject to the AS, IHX relations. The PBW isomorphism $\chi:\mathcal{B}\to\mathcal{A}(\downarrow)$ is defined by
\begin{align*}
\mbox{$\begin{array}{c}
   \includegraphics[scale=0.25]{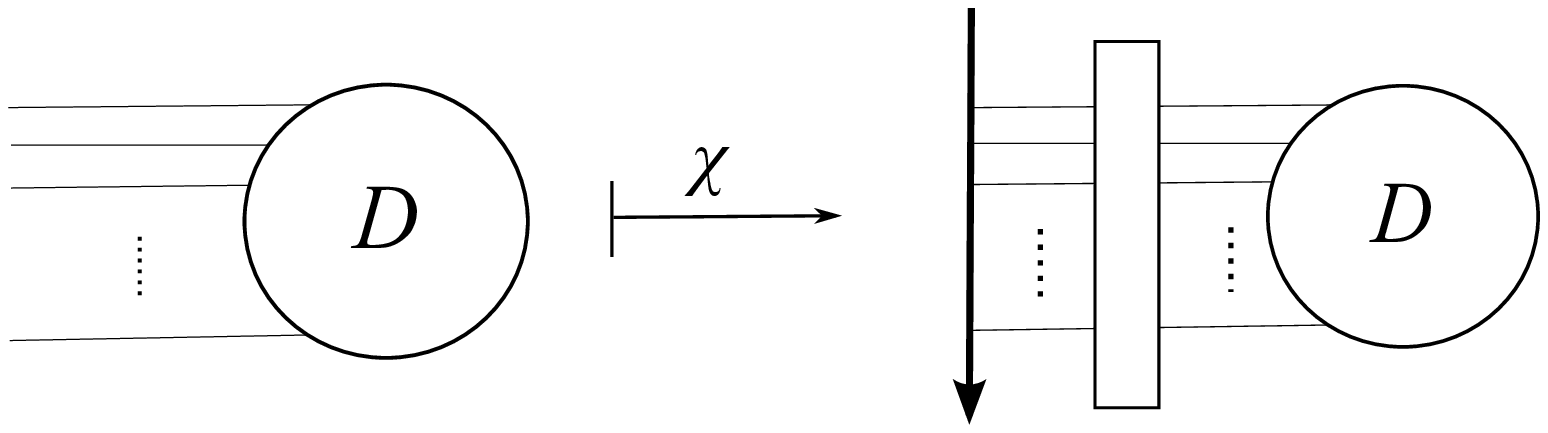}
   \end{array}$}
\end{align*}
for any diagram $D \in \mathcal{B}$, where the box means the symmetrizer,\\
\begin{align*}
\mbox{$\begin{array}{c}
   \includegraphics[scale=0.3]{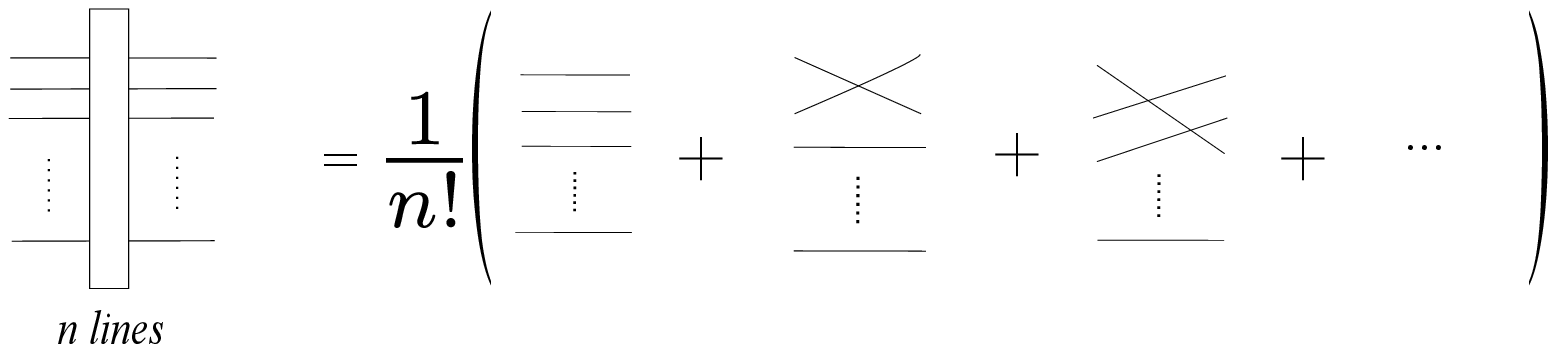}
   \end{array}$}.
\end{align*}
For a Jacobi diagram, we define the {\it degree} of it to be half the number of all vertices.\par
For a knot $K$, the {\it Kontsevich invariant} $Z(K)$ is defined to be in $\widetilde{\mathcal{A}(\downarrow)}$, and we also consider the Kontsevich invariant $\chi^{-1}Z(K)$ defined to be in $\widetilde{\mathcal{B}}$. Here, $\widetilde{\mathcal{A}(\downarrow)}$ (resp. $\widetilde{\mathcal{B}}$) denotes the completion of $\mathcal{A}(X)$ (resp. $\mathcal{B}$) with respect to the degree. It is known that each coefficient of the Kontsevich invariant is a Vassiliev invariant. It is also  known that $\chi^{-1}Z(K)$ can be presented as $\chi^{-1}Z(K)=\exp(\beta)$. Here, $\beta$ is an element in $\widetilde{\mathcal{B}_{\text{conn}}}$, the completion of the quotient $\mathbb{Q}$-vector space spanned by connected open Jacobi diagrams subject to the AS, IHX relations, and the exponential map is with respect to the product in $\widetilde{\mathcal{B}}$ defined by disjoint union of open Jacobi diagrams. Thus, we can consider $\log\big(\chi^{-1}Z(K)\big)\in\widetilde{\mathcal{B}_{\text{conn}}}$. For more details, see for example \cite{Oh1,Oh2}. For a knot $K$, we define {\it $n$-loop Kontsevich invariant} by $Z^{(n)}(K):=\iota_n\Big(\log\big(\chi^{-1}Z(K)\big)\Big)$, where $\iota_n:\widetilde{\mathcal{B}_{\text{conn}}}\to\widetilde{\mathcal{B}_{\text{conn}}^{(n)}}$ is the natural restriction map.

\section{The $n$-loop Kontsevich invariant of knots having the same Alexander polynomial}
\label{sec3}

In this section, we state the main results of this paper.\par
For $\Delta(t)\in\mathcal{Z}$, let $\mathcal{K}^{\Delta(t)}$ be a set of all knots whose values of Alexander polynomial are $\Delta(t)$, and let $\mathcal{V}(n,\Delta(t))$ be the subspace of $\widetilde{\mathcal{B}_{\text{conn}}^{(n)}}$ generated by the set $Z^{(n)}(\mathcal{K}^{\Delta(t)})$. We can show that $\mathcal{V}(n,\Delta(t))$ is infinite dimensional, see Appendix for its proof. This means that $Z^{(n)}(\mathcal{K}^{\Delta(t)})$ has the information of infinitely many Vassiliev invariants of those knots. Then, we consider the following filtration of the subspaces, 
\begin{align*}
\mathcal{V}(n,1,\Delta(t))\subset \mathcal{V}(n,2,\Delta(t))\subset\cdots\subset \mathcal{V}(n,\infty,\Delta(t)):=\bigcup_{g\geq 1}\mathcal{V}(n,g,\Delta(t))=\mathcal{V}(n,\Delta(t)).
\end{align*} 
Here $\mathcal{V}(n,g,\Delta(t))$ denotes the subspace of $\widetilde{\mathcal{B}_{\text{conn}}^{(n)}}$ generated by the set $Z^{(n)}(\mathcal{K}_{\leq g}^{\Delta(t)})$, where $\mathcal{K}_{\leq g}^{\Delta(t)}$ denotes the set of all genus $\leq g$ knots whose values of Alexander polynomial are $\Delta(t)$. It can be shown that the set $\mathcal{K}_{\leq g}^{\Delta(t)}$ is an infinite set unless it is $\emptyset$, see Appendix for its proof. Note that if $\text{deg}\Delta(t)>g$, we have $\mathcal{K}_{\leq g}^{\Delta(t)}=\emptyset$, since the breadth of the Alexander polynomial of any genus $\leq g$ knot is less than or equal to $2g$. Thus, the above filtration can be written as 
\begin{align*}
\emptyset=\cdots=\emptyset\subset \mathcal{V}(n,m,\Delta(t))\subset \mathcal{V}(n,m+1,\Delta(t))\subset\cdots\subset \mathcal{V}(n,\infty,\Delta(t)),
\end{align*}
where $m=\text{deg}\Delta(t)$.

\begin{thm}
\label{thm1}
For any integer $n\geq 2$, $g\geq 1$ and $\Delta(t)\in\mathcal{Z}$, the subspace $\mathcal{V}(n,g,\Delta(t))$ is finite dimensional. 
In particular, there exist finitely many elements $\beta_1,\cdots,\beta_d\in\widetilde{\mathcal{B}_{\text{conn}}^{(n)}}$ such that 
\begin{align*}
Z^{(n)}|_{\mathcal{K}_{\leq g}^{\Delta(t)}}=c_1\beta_1+\cdots+c_d\beta_d,
\end{align*}
where $c_1,\cdots,c_d:\mathcal{K}_{\leq g}^{\Delta(t)}\to\mathbb{Q}$ are (restrictions of) Vassiliev invariants.
\end{thm}

We denote $d(n,g,\Delta(t)):=\text{dim}\mathcal{V}(n,g,\Delta(t))$. Note that $d(n,g,\Delta(t))=0$ for any $\Delta(t)\in\mathcal{Z}$ with $\text{deg}\Delta(t)>g$.

We present $Z^{(n)}|_{\mathcal{K}_{\leq g}^{\Delta(t)}}=\sum_{j\geq 1}v_jD_j$, where $\{D_j\}_{j\geq 1}\subset\mathcal{B}_{\text{conn}}^{(n)}$, and $v_j:\mathcal{K}_{\leq g}^{\Delta(t)}\to\mathbb{Q}$ is a (restriction of) Vassiliev invariant.

\begin{cor}
\label{corvv}
For any integer $n\geq 2$, $g\geq 1$ and $\Delta(t)\in\mathcal{Z}$, we fix above $\{D_j\}_{j\geq 1}$ and $\{v_j\}_{j\geq 1}$, and we put $d:=d(n,g,\Delta(t))$. Then, there exist $d$ elements $v_{j_1},\cdots,v_{j_d}\in\{v_j\}_{j\geq 1}$, such that
\begin{align*}
\{v_j\}_{j\geq 1}\subset\text{span}_{\mathbb{Q}}\{v_{j_1},\cdots,v_{j_d}\}.
\end{align*}
In particular, for any $K_1,K_2\in\mathcal{K}_{\leq g}^{\Delta(t)}$, we have
\begin{align*}
v_{j_k}(K_1)=v_{j_k}(K_2)\quad(\text{for all $k=1,\cdots,d$})\iff Z^{(n)}(K_1)=Z^{(n)}(K_2).
\end{align*}
\end{cor}

Corollary \ref{corvv} immediately follows from Theorem \ref{thm1}.

\begin{rem}
Note that a choice of elements $v_{j_1},\cdots,v_{j_d}\in\{v_j\}_{j\geq 1}$ in Corollary \ref{corvv} is not unique.
\end{rem}

The above statements indicate that $Z^{(n)}|_{\mathcal{K}_{\leq g}^{\Delta(t)}}$ has the information of only finitely many Vassiliev invariants, and we can present each coefficient of it by a linear sum of those Vassiliev invariants. Thus, roughly speaking, $d(n,g,\Delta(t))$ represents ``how many Vassiliev invariants determine $Z^{(n)}|_{\mathcal{K}_{\leq g}^{\Delta(t)}}$''. We consider the following problem.

\begin{prob}
For any integer $n\geq 2$, $g\geq 1$ and $\Delta(t)\in\mathcal{Z}$, determine $\mathcal{V}(n,g,\Delta(t))$ and $d(n,g,\Delta(t))$.
\end{prob}

It seems hard to solve this problem completely. Here, we solve this for the following simplest case.

\begin{thm}
\label{t21d}
For $\Delta(t)=1+a(t+t^{-1}-2)\in\mathcal{Z}$, where $a\in\mathbb{Z}$, we have
\begin{align*}
\mathcal{V}(2,1,\Delta(t))=\text{span}_{\mathbb{Q}}(\theta_1,\theta_2).
\end{align*}
Here
\begin{align*}
\theta_1&={\mbox{$\begin{array}{c}
   \includegraphics[scale=0.35]{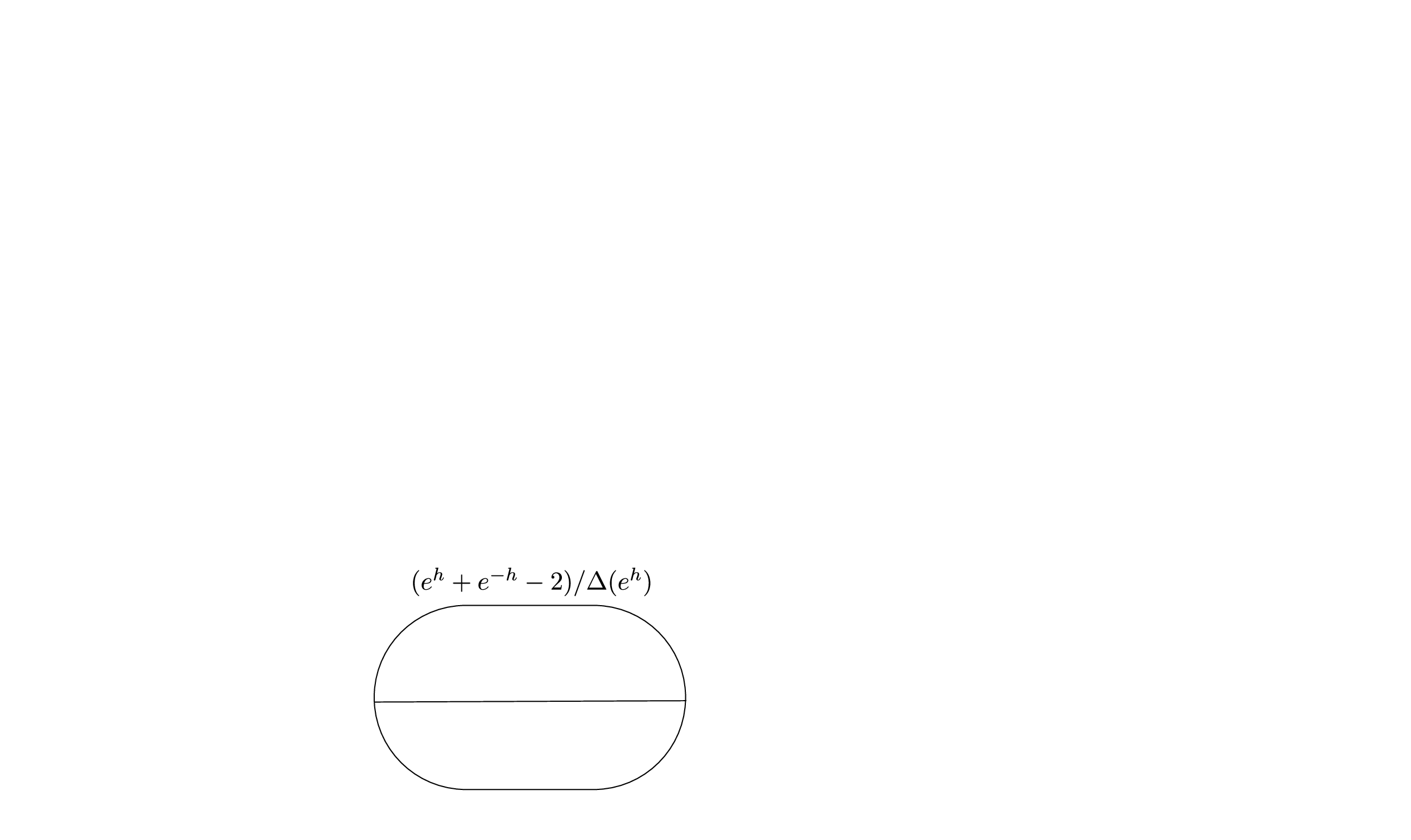}
   \end{array}$}}\in\widetilde{\mathcal{B}_{\text{conn}}^{(2)}},\\*
\theta_2&={\mbox{$\begin{array}{c}
   \includegraphics[scale=0.35]{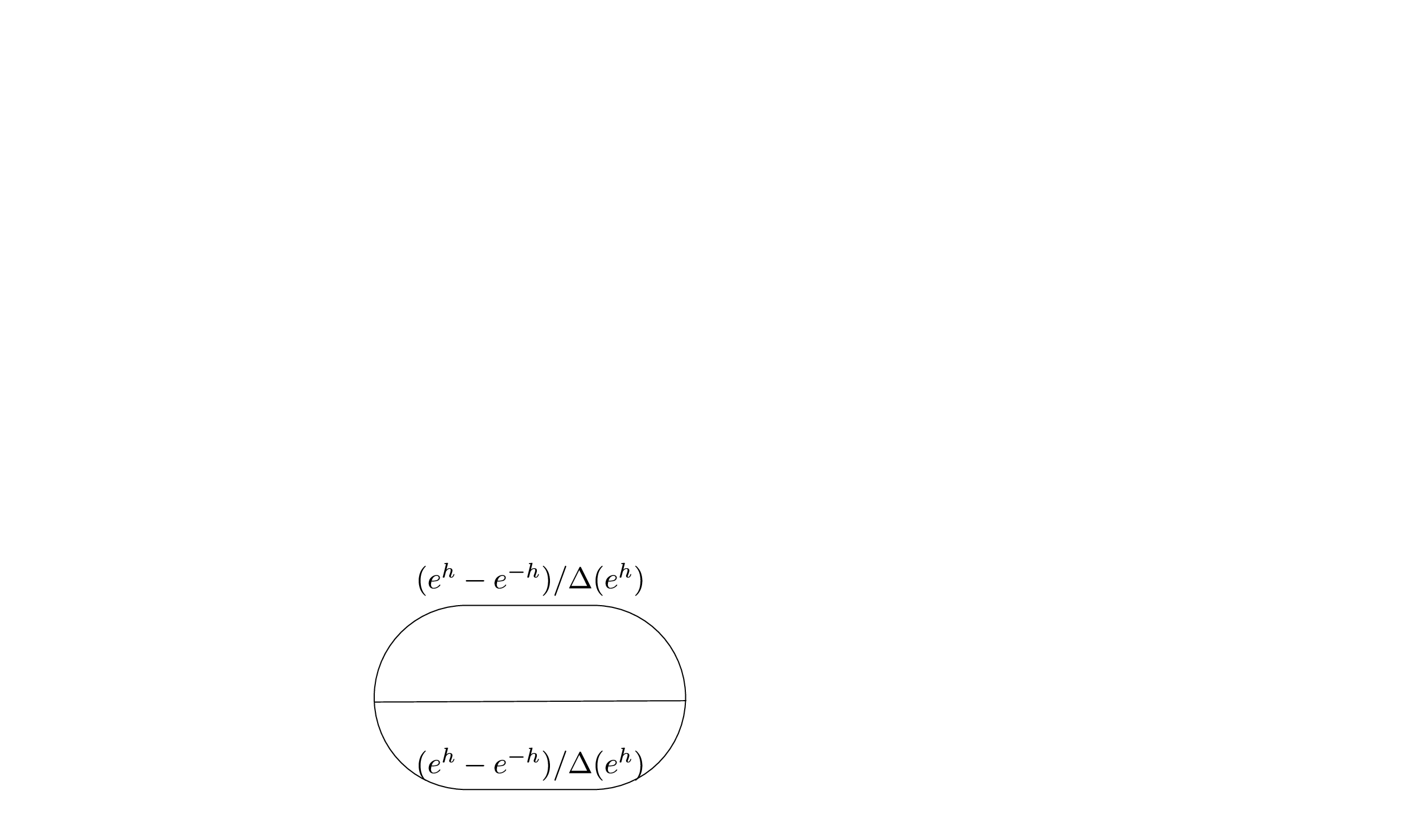}
   \end{array}$}}+\left(\df{4}{3}a-\df{1}{3}\right){\mbox{$\begin{array}{c}
   \includegraphics[scale=0.35]{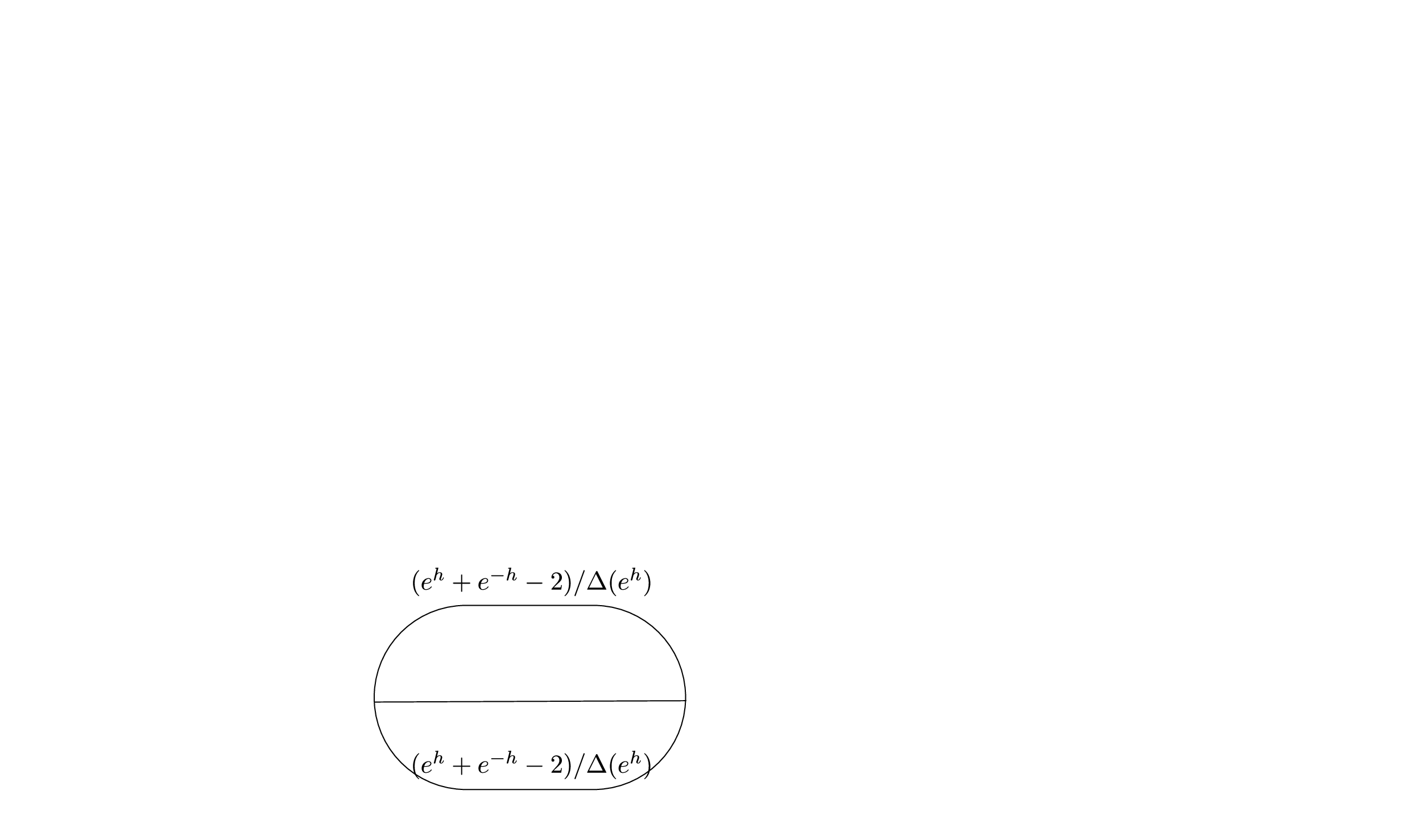}
   \end{array}$}}\in\widetilde{\mathcal{B}_{\text{conn}}^{(2)}}.
\end{align*}
where we regard $(e^h+e^{-h}-2)/\Delta(e^h)$ and $(e^h-e^{-h})/\Delta(e^h)$ as power series in $h$, and we define a labeling on one side of an edge of a Jacobi diagram by an infinite series in $f(h)=c_0+c_1h+\cdots$ by\\
\begin{align*}
\mbox{$\begin{array}{c}
   \includegraphics[scale=0.3]{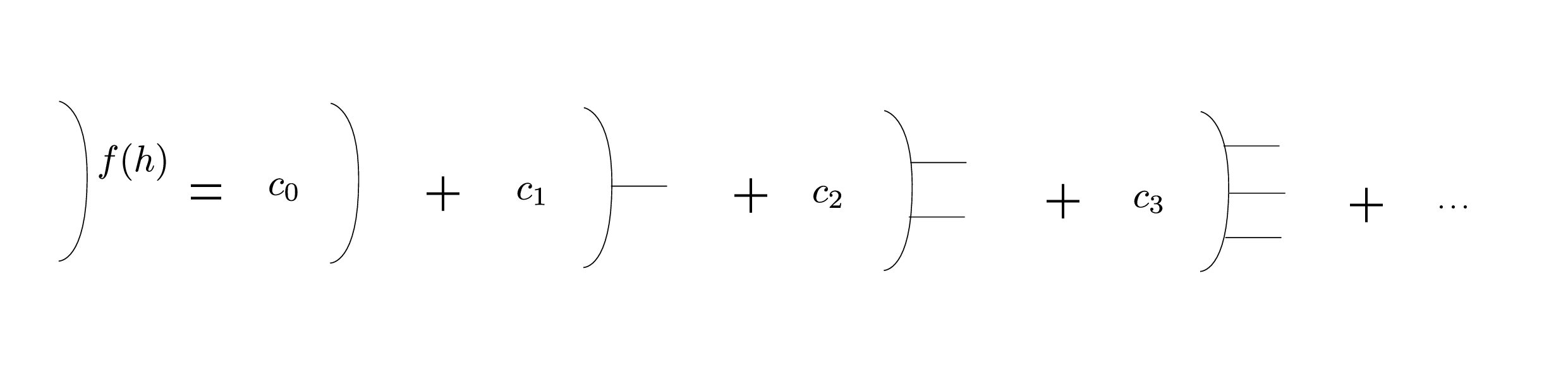}
   \end{array}$}\in\widetilde{\mathcal{B}},
\end{align*}
(for details, see Section \ref{sec4}). In particular, for any $\Delta(t)\in\mathcal{Z}$, we have
\begin{align*}
d(2,1,\Delta(t))=\left\{
\begin{array}{ll}
2&\text{when $\Delta(t)=1+a(t+t^{-1}-2)$ for $a\in\mathbb{Z}$}\\
0&\text{otherwise}
\end{array}
\right.
\end{align*}
\end{thm}

Theorem \ref{t21d} can be regarded as a corollary of Theorem 3.1 of \cite{Oh4}.\par
A formula of $Z^{(2)}$ for genus 1 knots is given in \cite{Oh4}, in terms of finite type invariants of representation tangles of these knots, see Section \ref{sec5} for representation tangles of knots. As a corollary of Theorem \ref{t21d}, we can give another formula of $Z^{(2)}$ for genus 1 knots, in terms of Vassiliev invariants of these knots.

\begin{cor}
\label{2loop}
Let $K$ be a genus 1 knot whose Alexander polynomial is $1+a(t+t^{-1}-2)$ for $a\in\mathbb{Z}$, and we present the first two term of $Z^{(2)}(K)$ as
\begin{align*}
Z^{(2)}(K)=b_1{\mbox{$\begin{array}{c}
   \includegraphics[scale=0.2]{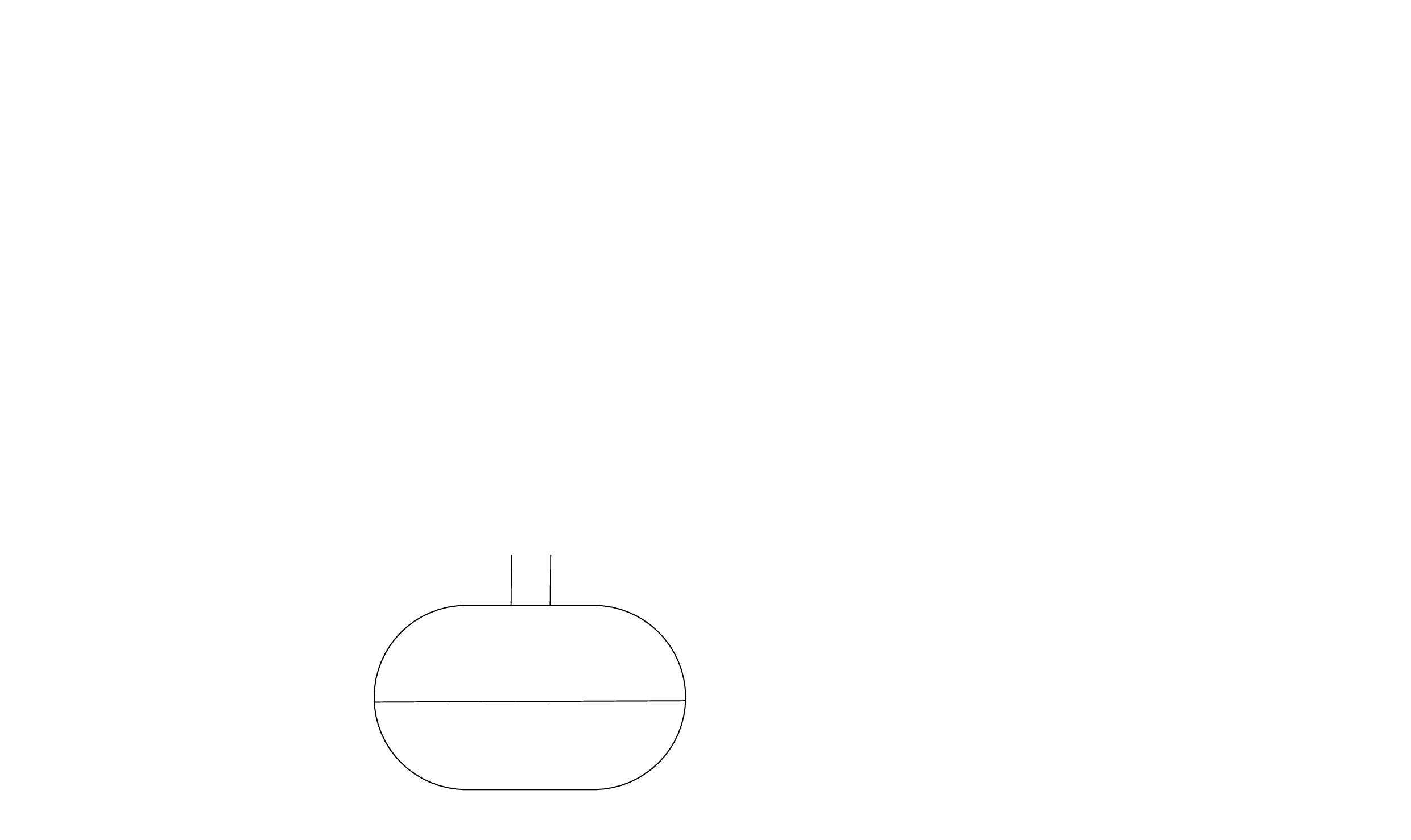}
   \end{array}$}}+b_2{\mbox{$\begin{array}{c}
   \includegraphics[scale=0.2]{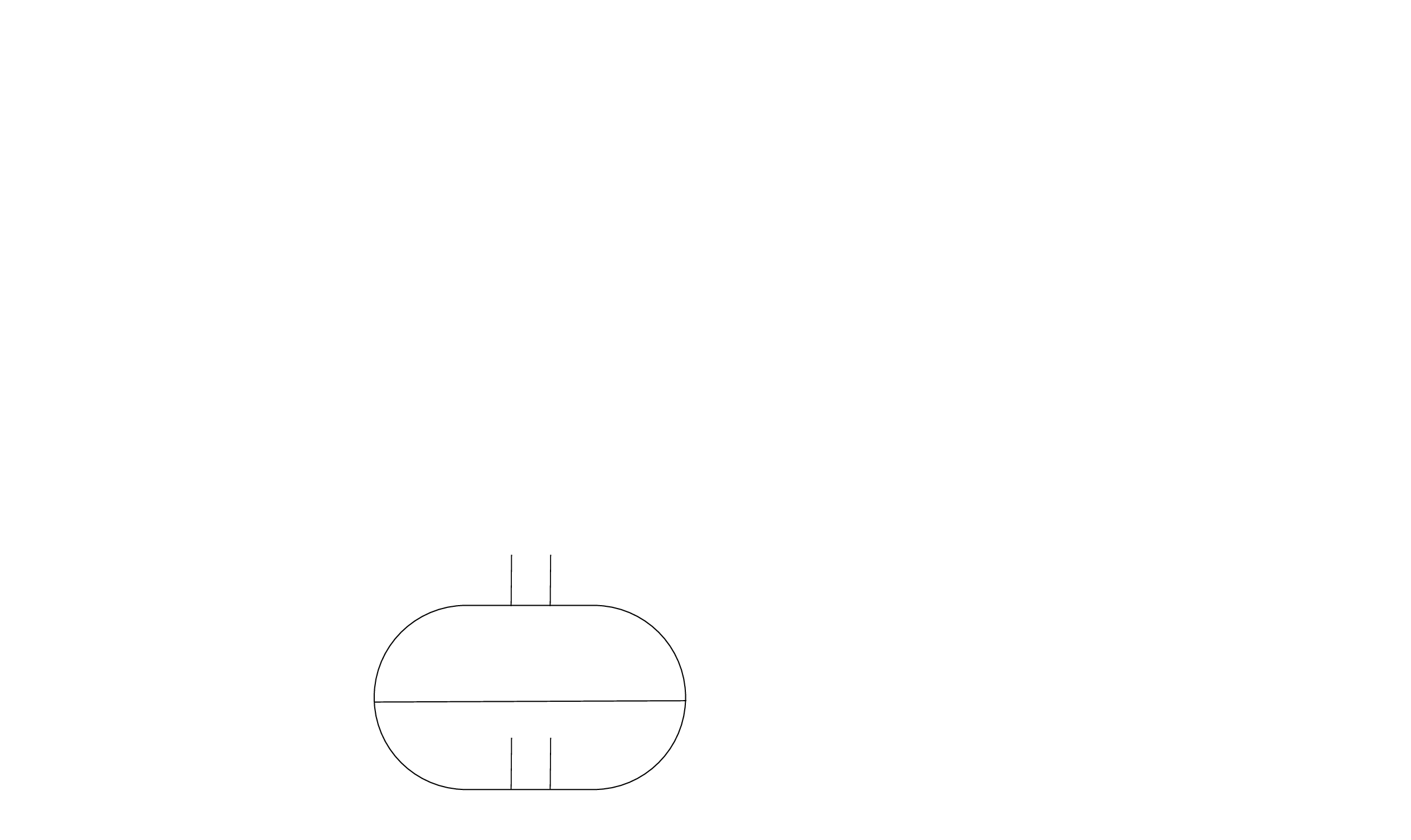}
   \end{array}$}}+\text{(terms with higher degree)},
\end{align*}
where $b_1,b_2$ are Vassiliev invariants of degree $3,5$, respectively. Then, we have
\begin{align*}
Z^{(2)}(K)=&\df{(28a-5)b_1+6b_2}{16a-4}{\mbox{$\begin{array}{c}
   \includegraphics[scale=0.25]{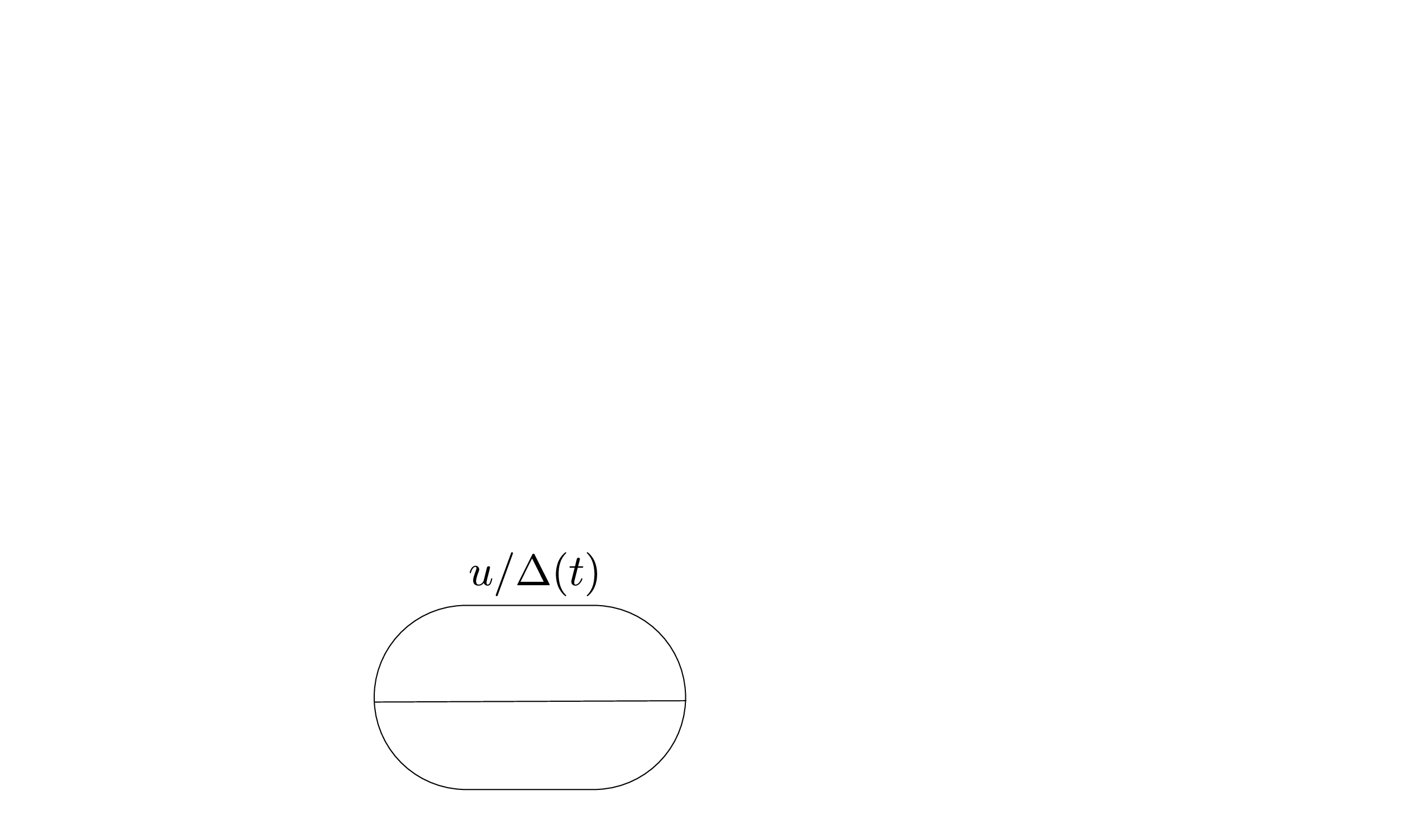}
   \end{array}$}}+\df{(12a-1)b_1+6b_2}{32a-8}{\mbox{$\begin{array}{c}
   \includegraphics[scale=0.25]{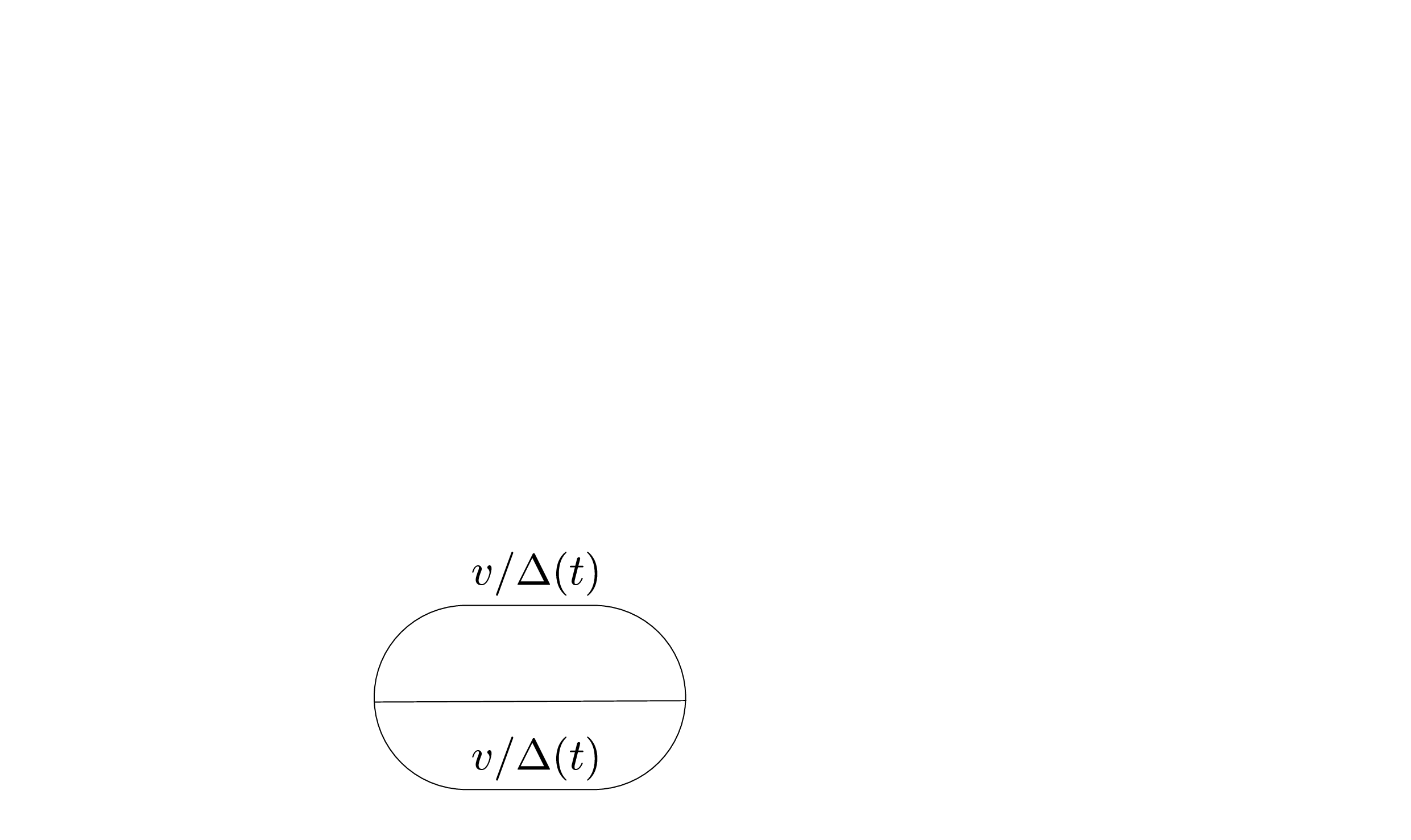}
   \end{array}$}}\\*
&+\df{(12a-1)b_1+6b_2}{24}{\mbox{$\begin{array}{c}
   \includegraphics[scale=0.25]{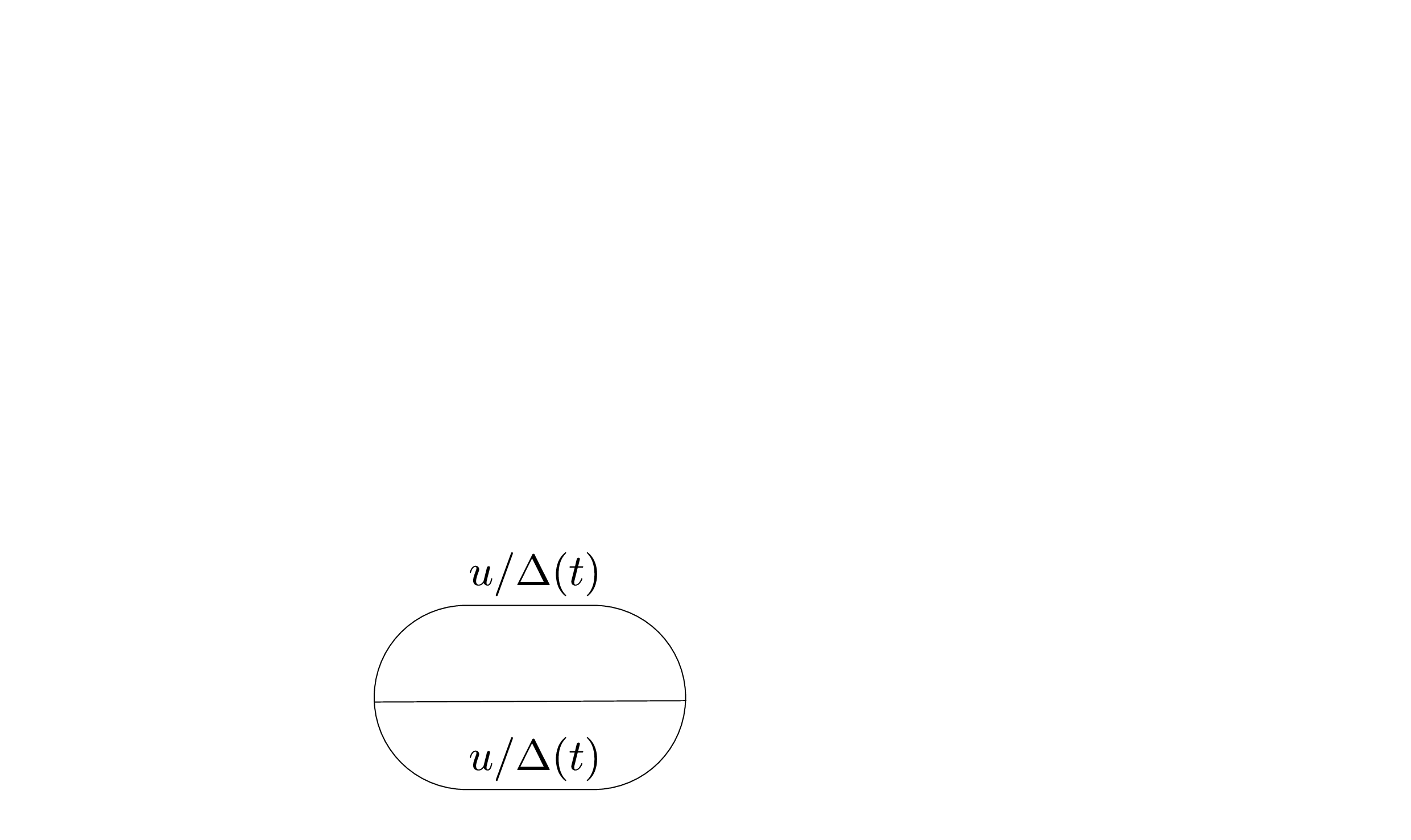}
   \end{array}$}},
\end{align*}
where $u=t+t^{-1}-2$, $v=t-t^{-1}$, and $t=e^h$.
\end{cor}

For other simple cases, we give some upper bounds.

\begin{thm}
\label{thm6}
For any integer $g\geq 1$ and $\Delta(t)\in\mathcal{Z}$, we have
\begin{align*}
\mathcal{V}(2,g,\Delta(t))\subset\text{span}_{\mathbb{Q}}\{\Theta_m^n\mid m,n\in\mathbb{Z}, n\geq 1, 0\leq 2m\leq n\leq 2g\}.
\end{align*}
Here
\begin{align*}
\Theta_m^n={\mbox{$\begin{array}{c}
   \includegraphics[scale=0.25]{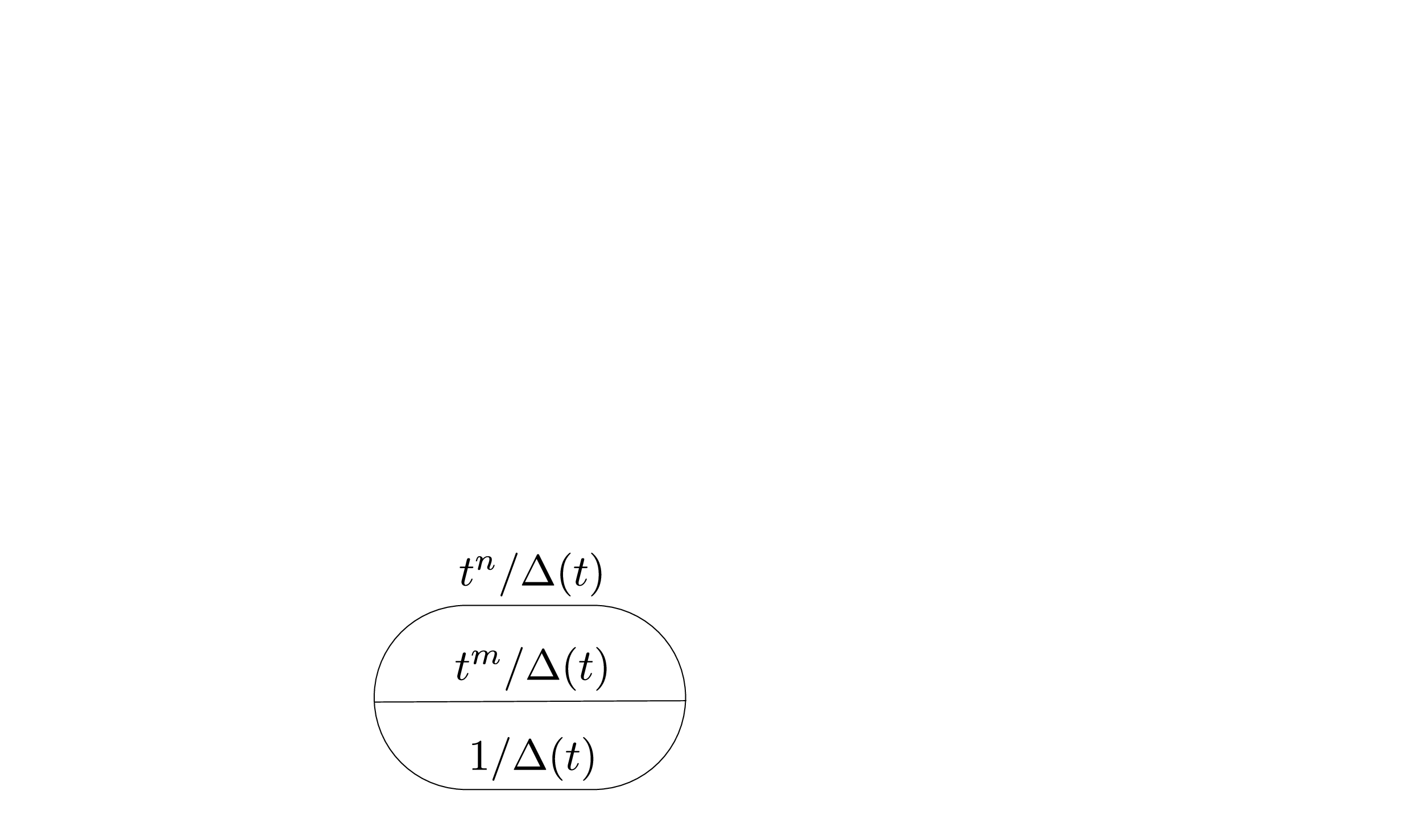}
   \end{array}$}}-{\mbox{$\begin{array}{c}
   \includegraphics[scale=0.25]{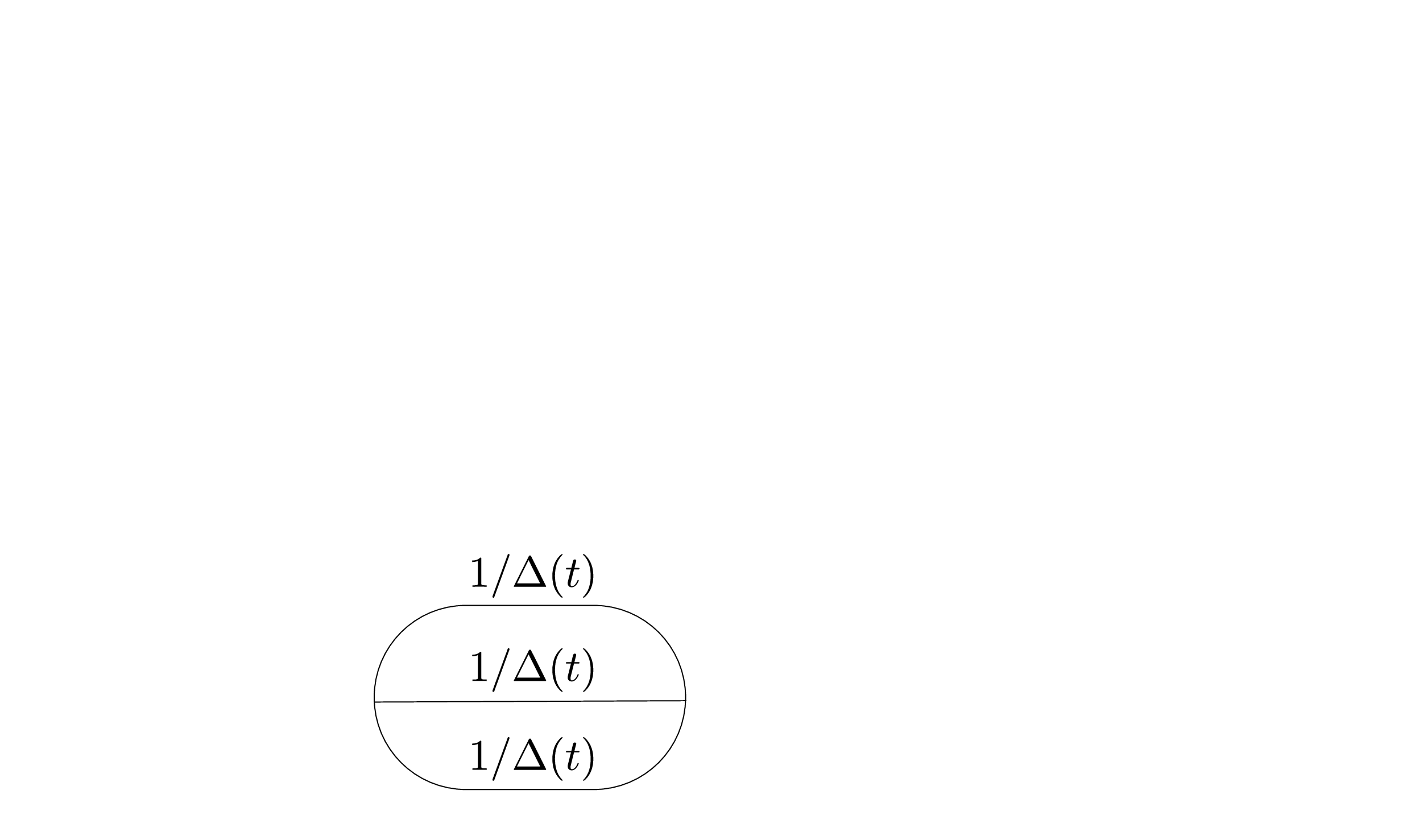}
   \end{array}$}},\quad\text{and  $t=e^h$}.
\end{align*}
In particular, we have $d(2,g,\Delta(t))\leq g^2+2g$.
\end{thm}

Theorem \ref{thm6} can be regarded as a corollary of Theorem 4.7 of \cite{Oh4}.

\begin{rem}
The equality in Theorem \ref{thm6} does not hold even $g=1$ case. Hence we hope to obtain better evaluations or calculations.
\end{rem}

\begin{thm}\cite{Ya3}
\label{Ya3}
We have
\begin{align*}
\mathcal{V}(3,1,1)\subset\text{span}_{\mathbb{Q}}X.
\end{align*}
Here
\begin{align*}
X=\{&{\mbox{$\begin{array}{c}
   \includegraphics[scale=0.25]{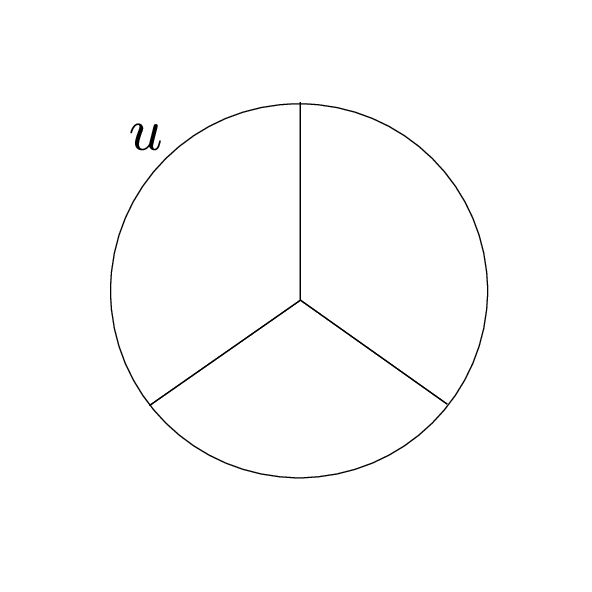}
   \end{array}$}},{\mbox{$\begin{array}{c}
   \includegraphics[scale=0.25]{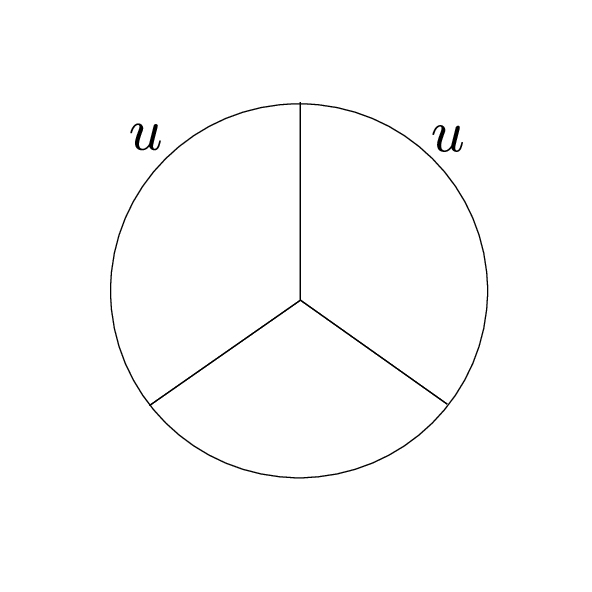}
   \end{array}$}},{\mbox{$\begin{array}{c}
   \includegraphics[scale=0.25]{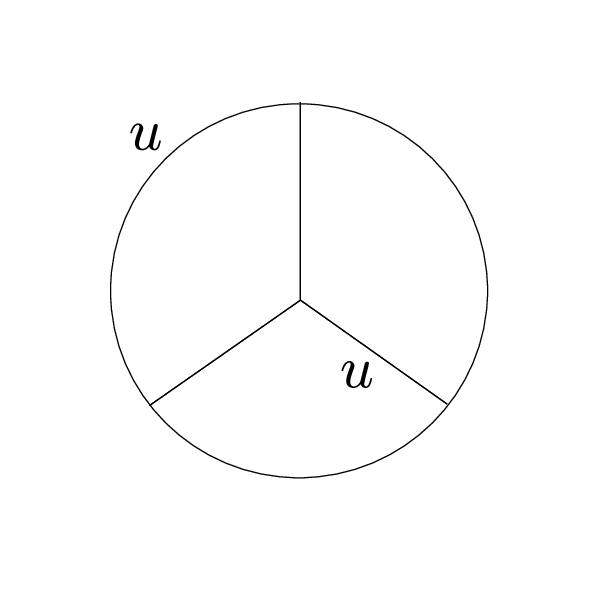}
   \end{array}$}},{\mbox{$\begin{array}{c}
   \includegraphics[scale=0.25]{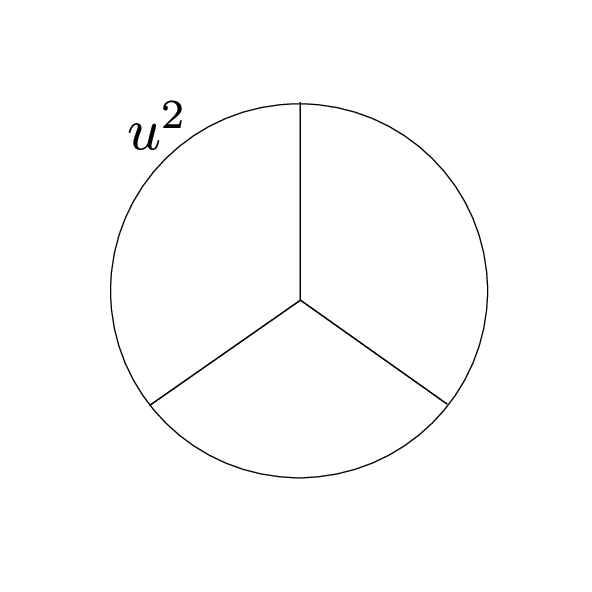}
   \end{array}$}},{\mbox{$\begin{array}{c}
   \includegraphics[scale=0.25]{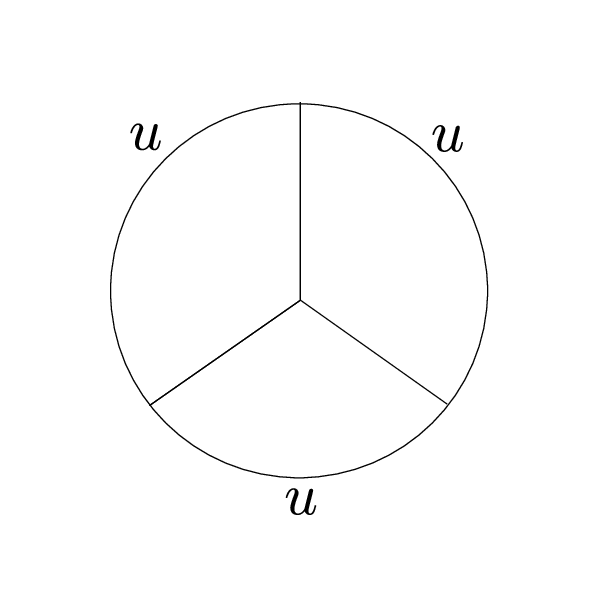}
   \end{array}$}},{\mbox{$\begin{array}{c}
   \includegraphics[scale=0.25]{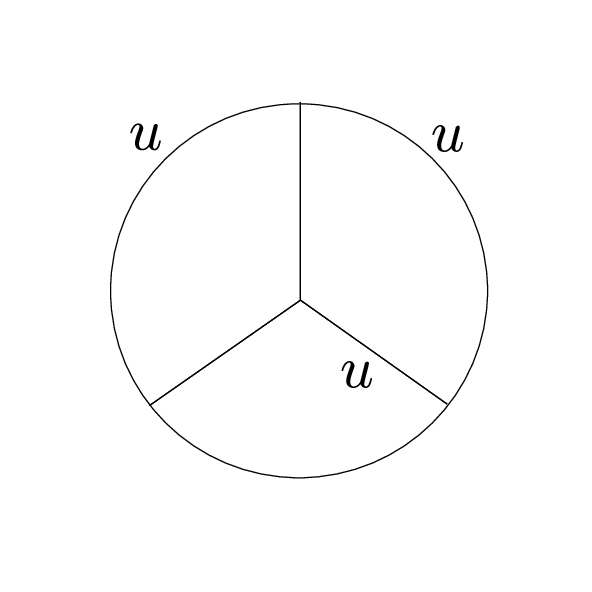}
   \end{array}$}},{\mbox{$\begin{array}{c}
   \includegraphics[scale=0.25]{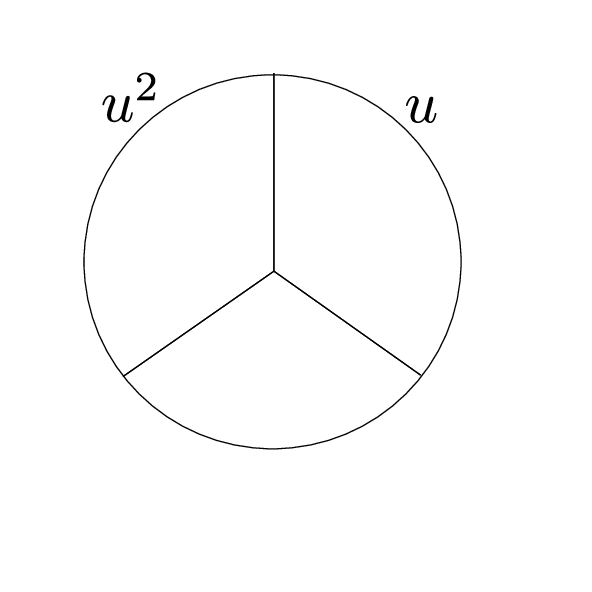}
   \end{array}$}},\\
&{\mbox{$\begin{array}{c}
   \includegraphics[scale=0.25]{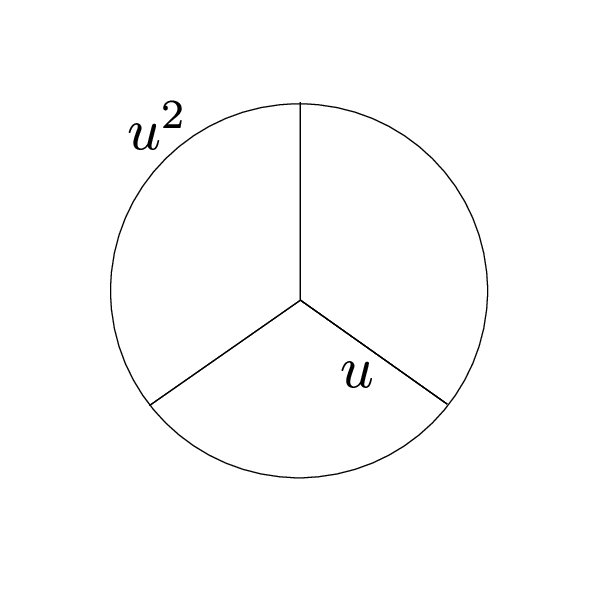}
   \end{array}$}},{\mbox{$\begin{array}{c}
   \includegraphics[scale=0.25]{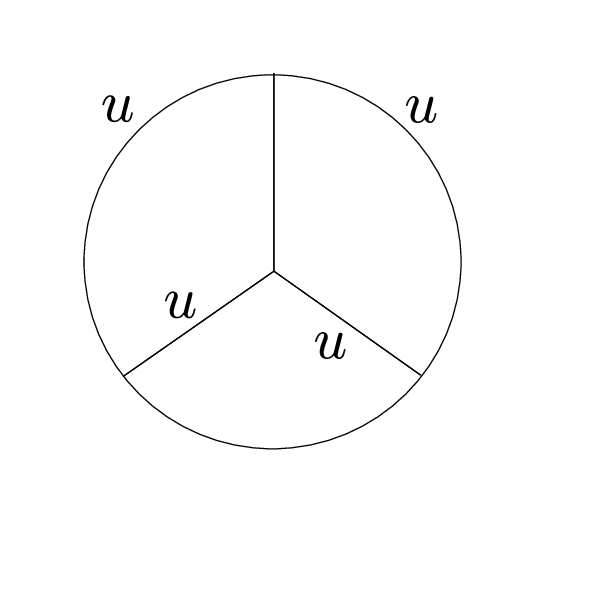}
   \end{array}$}},{\mbox{$\begin{array}{c}
   \includegraphics[scale=0.25]{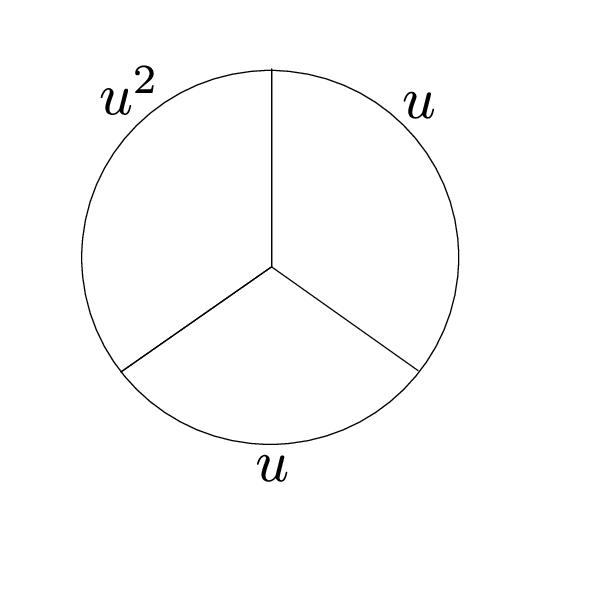}
   \end{array}$}},{\mbox{$\begin{array}{c}
   \includegraphics[scale=0.25]{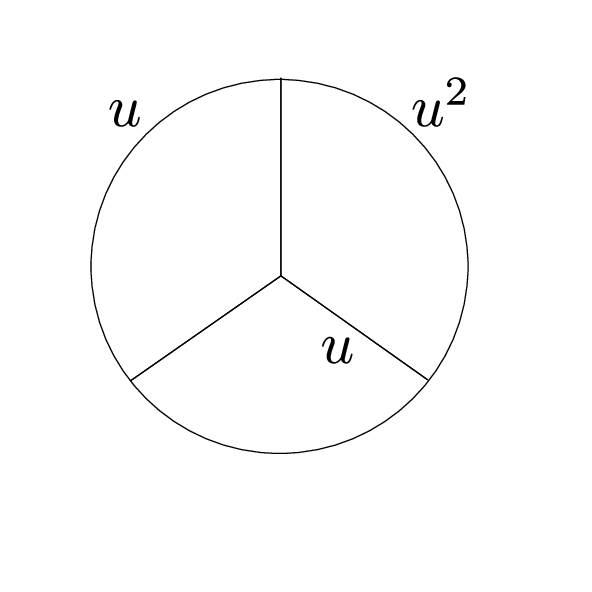}
   \end{array}$}}\},
\end{align*}
where $u=t+t^{-1}-2$ and $t=e^h$. In particular, we have $d(3,1,1)\leq 11$.
\end{thm}

\begin{rem}
If we can determine $l:=\text{max}\{\text{degree of }v_{j_k}\mid k=1,\cdots,d\}$, where $\{v_{j_1},\cdots,v_{j_d}\}$ is given in Corollary \ref{corvv}, we can obtain the lower bound of the genus of knots with trivial $Z^{(k)}$ for all $1\leq k\leq n-1$ and trivial degree $\leq l$ parts of $Z^{(n)}$. Note that we can obtain such knots by, for example, clasper surgeries, see Section \ref{sec4} for clasper surgeries.
\end{rem}

\section{A review of the rational version of the Aarhus integral and clasper surgeries}
\label{sec4}

\subsection{A review of the rational version of the Aarhus integral}

In this section, we briefly review the loop expansion of the Kontsevich invariant of knots and the rational version of Aarhus integral. For details, see \cite{BGRT1,BGRT2,BGRT3,Ga1,Kri1}\par
We obtain the loop expansion of the Kontsevich invariant of a knot $K$, as follows. We take its surgery presentation $K_0\cup L$ such that $K_0$ is the 0-framed unknot and $L$ is a $l$-components framed link, and the linking number of $K_0$ and each component of $L$ is equal to 0. This means that the pair $(S^3,K)$ is obtained from the pair $(S^3,K_0)$ by surgery along $L$. Let $X$ be a finite set with $|X|=l$. First, we compute $Z(K_0\cup L)\in\widetilde{\mathcal{A}(\bigsqcup_{\{h\}\cup X}S^1)}$. Here $\bigsqcup_{\{h\}\cup X}S^1$ is the disjoint union of $l+1$ oriented circles, each of which is labeled by each of elements in $\{h\}\cup X$ respectively, and $h$ corresponds to $K_0$ and each component of $X$ corresponds to each element of $L$ respectively. Second, we compute $\check{Z}(K_0\cup L)$, which is obtained from $Z(K_0\cup L)$ by connected-summing of $\nu=Z(\text{unknot})$ to each component labeled by an element of $X$. Third, we compute $\chi_{h}^{-1}\check{Z}(K_0\cup L)\in\widetilde{\mathcal{A}(\ast_h\sqcup\bigsqcup_XS^1)}$, where $\chi_h:\mathcal{A}(\ast_h\sqcup\bigsqcup_XS^1)\rightarrow\mathcal{A}(\downarrow_h\sqcup\bigsqcup_XS^1)\cong\mathcal{A}(\bigsqcup_{\{h\}\cup X}S^1)$ is defined as the PBW isomorphism. It is known, see \cite[Theorem 4]{BLT}, \cite[Corollary 5.0.8]{Kri1}, that
\begin{align*}
\chi_h^{-1}Z\left({\mbox{$\begin{array}{c}
   \includegraphics[scale=0.18]{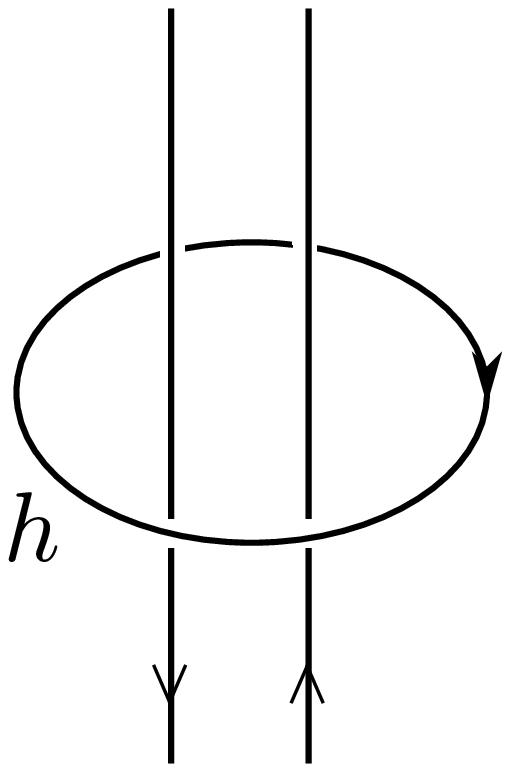}
 \end{array}$}}\right)
={\mbox{$\begin{array}{c}
   \includegraphics[scale=0.18]{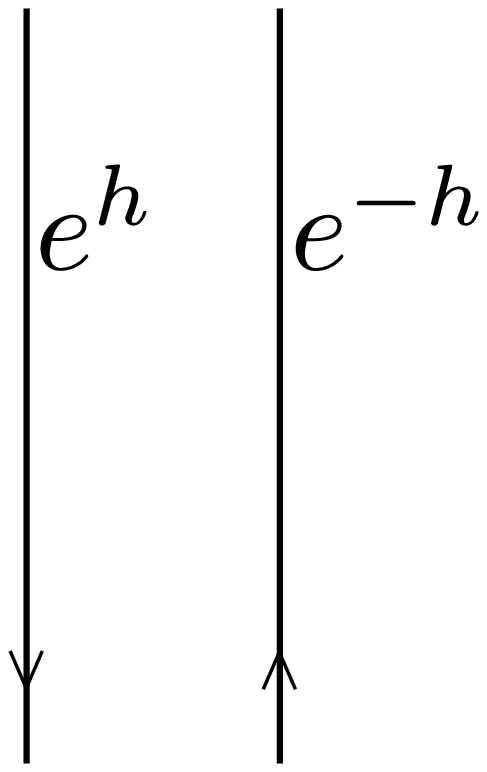}
 \end{array}$}}
\sqcup\chi^{-1}\nu.
\end{align*}
For simplicity, by putting $t=e^h$ and omitting $\chi^{-1}\nu$, we denote
\begin{align}
\label{xzt}
\chi_h^{-1}Z\left({\mbox{$\begin{array}{c}
   \includegraphics[scale=0.18]{longhopf2.ps}
 \end{array}$}}\right)
={\mbox{$\begin{array}{c}
   \includegraphics[scale=0.18]{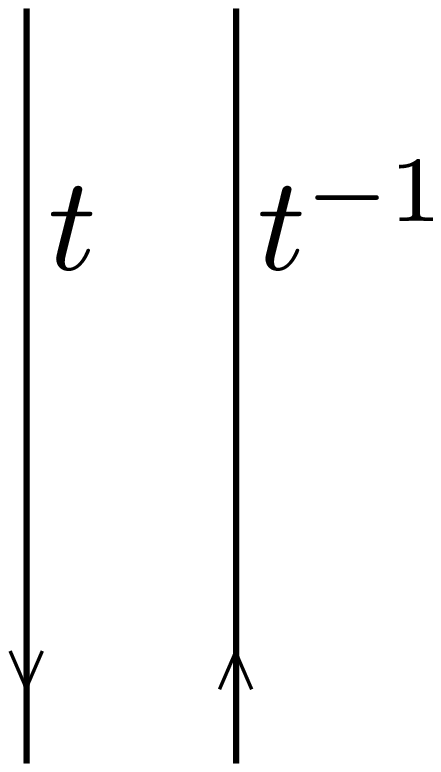}
 \end{array}$}}.
\end{align}
In the following of this paper, we use the following notation. We define a labeling on one side of an edge of a Jacobi diagram by an infinite series in $f(h)=c_0+c_1h+\cdots$ by\\
\begin{align}
\label{label}
\mbox{$\begin{array}{c}
   \includegraphics[scale=0.3]{labeling1.ps}
   \end{array}$}\in\widetilde{\mathcal{A}(\ast_h)}=\widetilde{\mathcal{B}}.
\end{align}
Thus, the labeling on one side of an edge of a Jacobi diagram by (a rational function of) $t$ means
\begin{align}
\label{tehh}
\mbox{$\begin{array}{c}
   \includegraphics[scale=0.3]{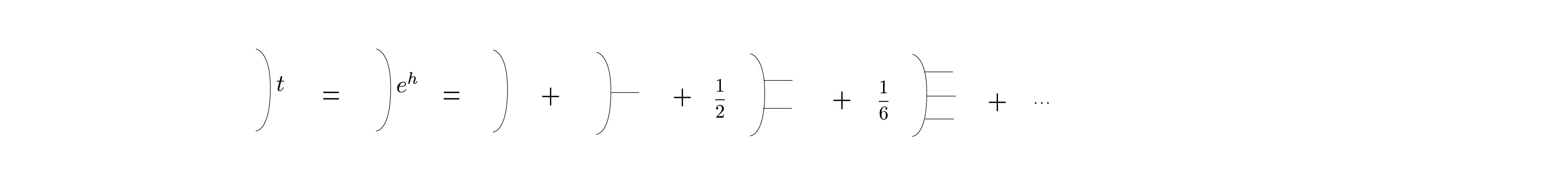}
   \end{array}$}\in\widetilde{\mathcal{A}(\ast_h)}=\widetilde{\mathcal{B}}.
\end{align}
Note that by the AS, IHX, and STU relations, we have
\begin{align}
\label{trel}
\mbox{$\begin{array}{c}
   \includegraphics[scale=0.3]{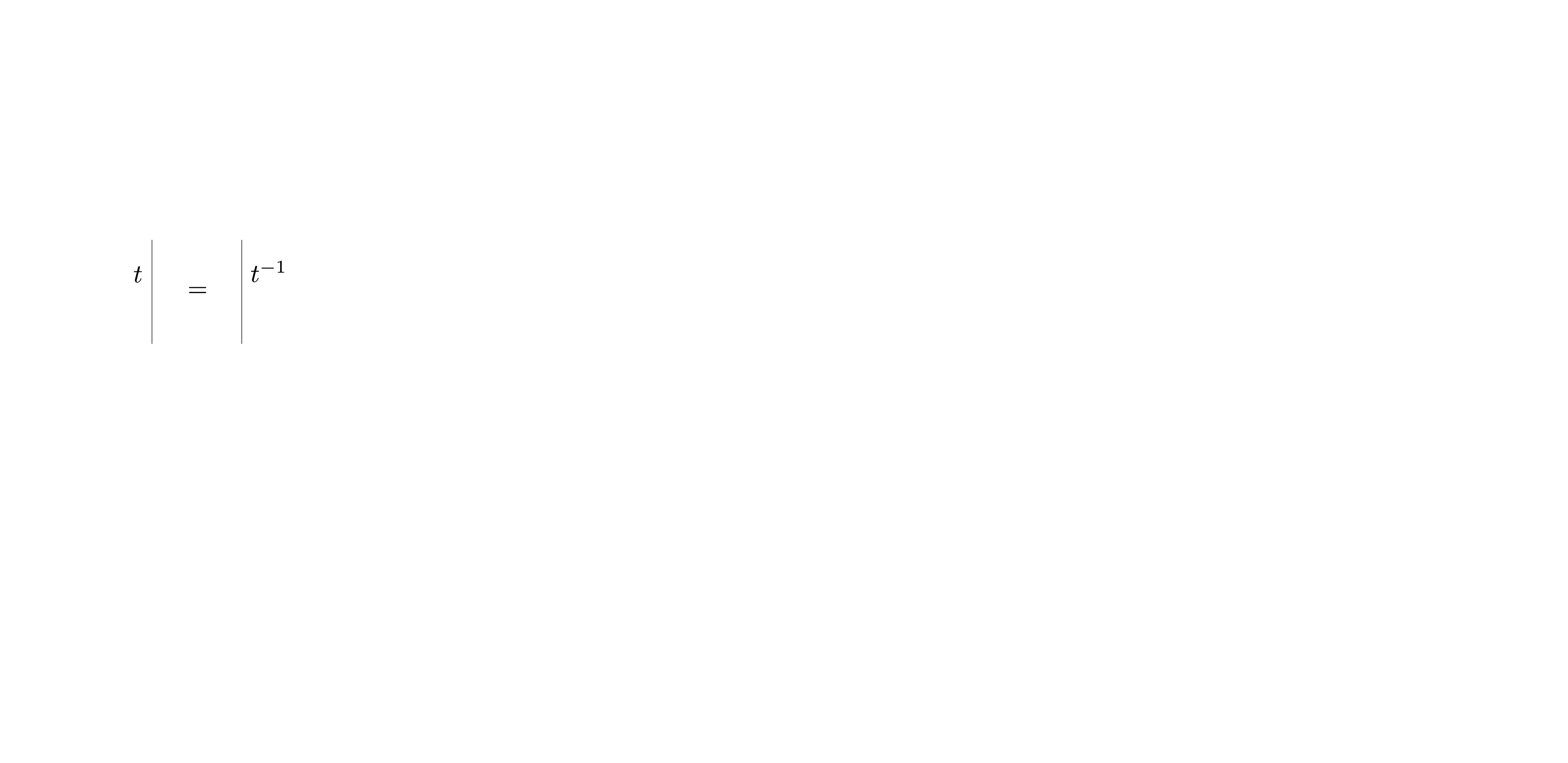}
   \end{array}$}\quad&,\qquad\mbox{$\begin{array}{c}
   \includegraphics[scale=0.3]{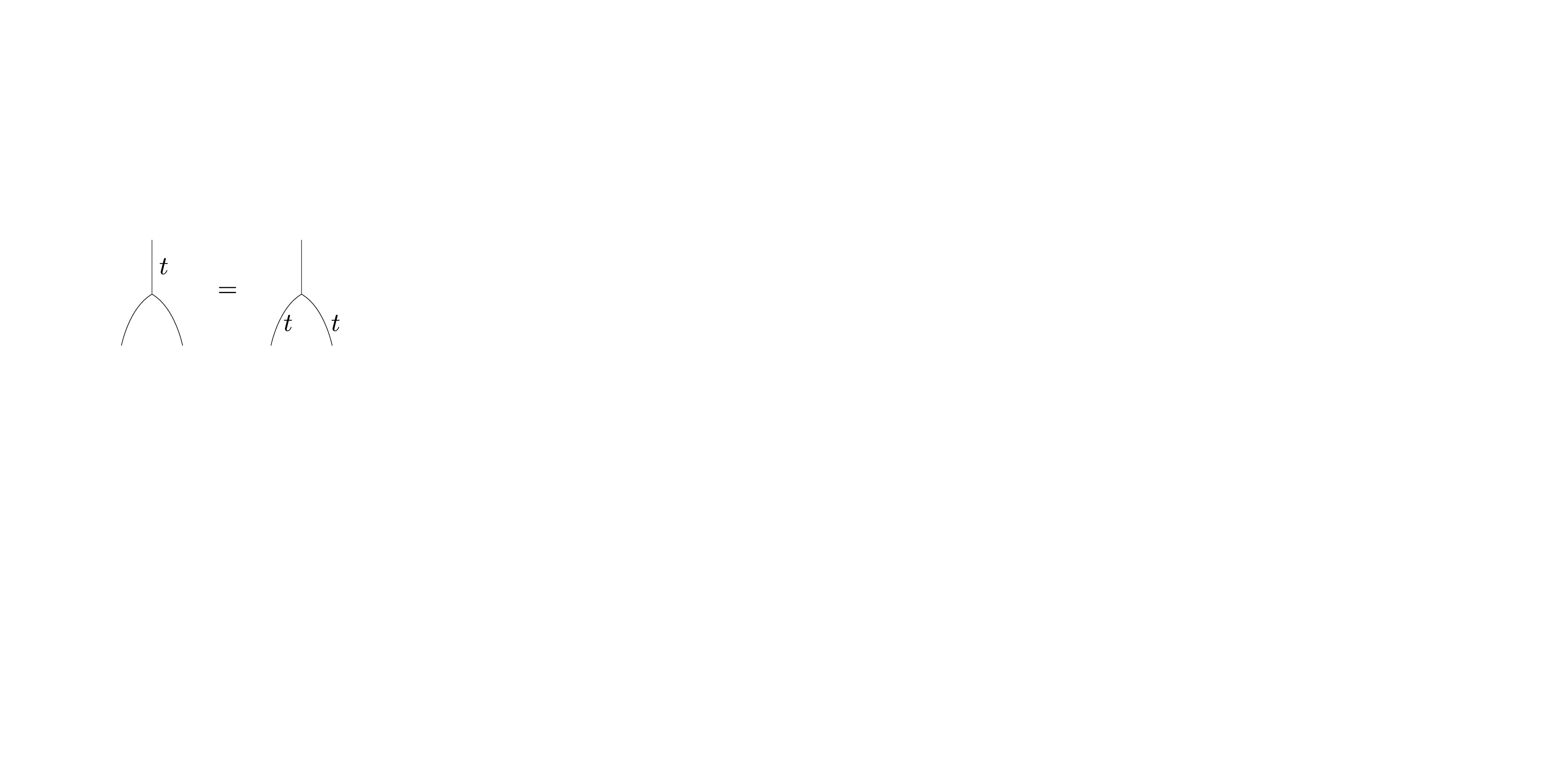}
   \end{array}$},\nonumber\\*
\mbox{$\begin{array}{c}
   \includegraphics[scale=0.3]{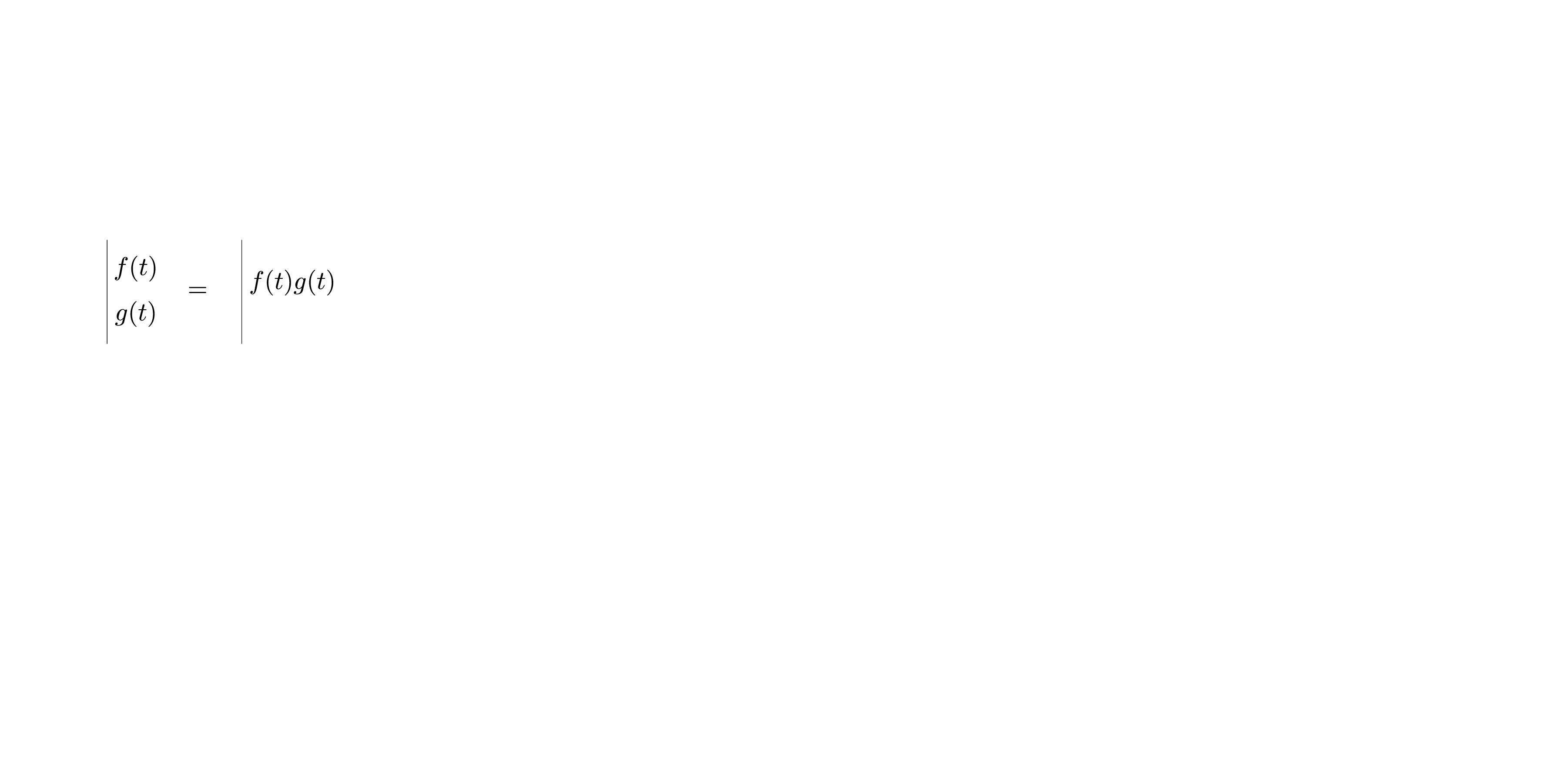}
   \end{array}$}\quad&,\qquad\mbox{$\begin{array}{c}
   \includegraphics[scale=0.3]{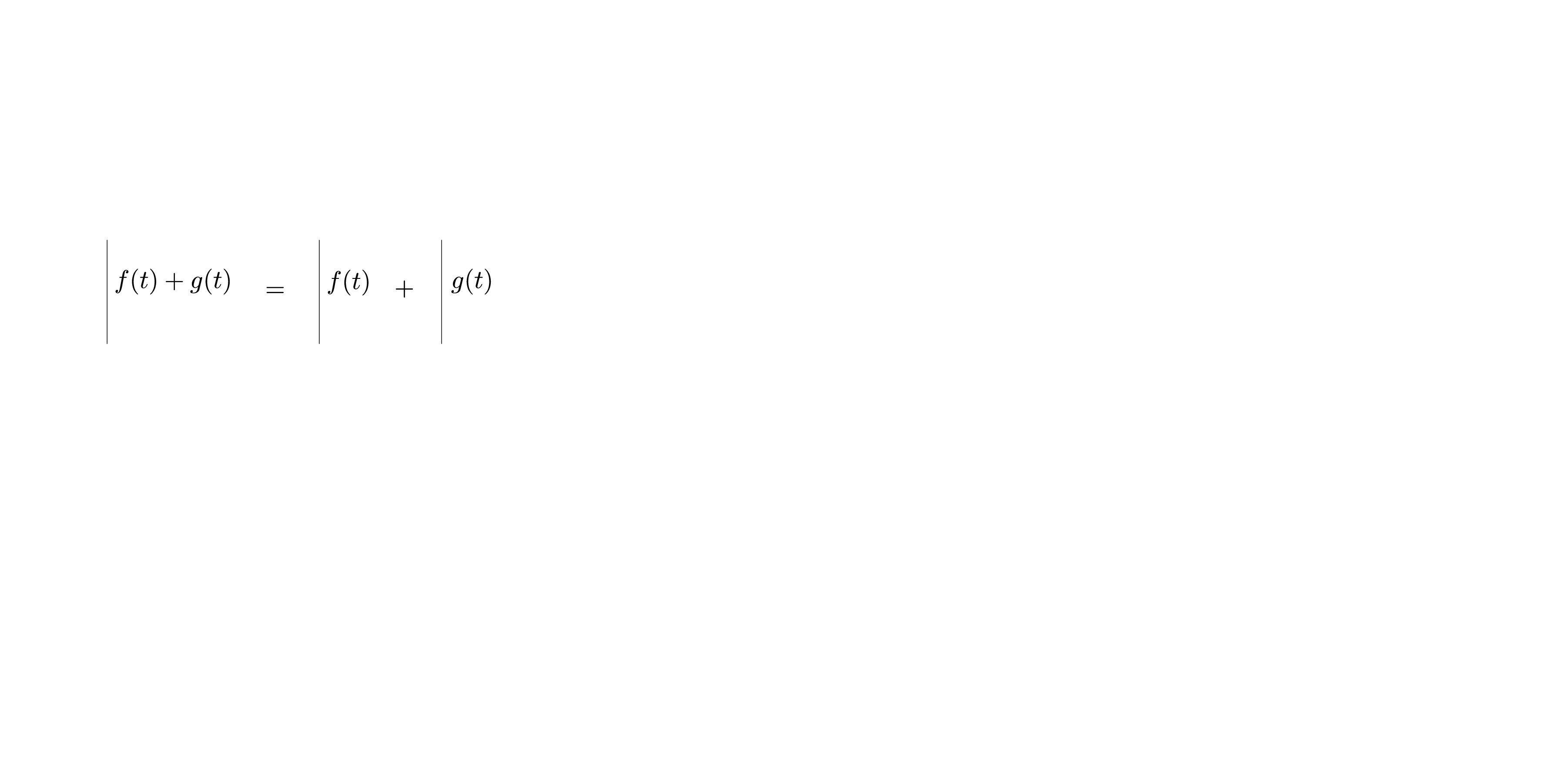}
   \end{array}$},
\end{align}
where the vertex appeared in the second relation is either a univalent vertex or a trivalent vertex. Forth, we compute $\chi^{-1}\check{Z}(K_0\cup L)\in\widetilde{\mathcal{A}(\ast_{\{h\}\cup X})}$, where $\mathcal{A}(\ast_{\{h\}\cup X})$ is the quotient space spanned by open Jacobi diagram whose univalent vertex is labeled by an element in $\{h\}\cup X$, subject to AS, IHX relations. We choose a disjoint union of the unknot and a string link $K_0\cup\check{L}$ whose closure is isotopic to $K_0\cup L$, and the map $\chi_X:\mathcal{A}(\ast_{\{h\}\cup X})\rightarrow\mathcal{A}(\ast_{h}\sqcup\bigsqcup_X\downarrow)$ is the map defined by the composition of all PBW isomorphisms for all elements in $X$. For simplicity, we denote $\chi_X^{-1}\chi_h^{-1}\check{Z}(K_0\cup L)$ by $\chi^{-1}\check{Z}(K_0\cup L)$. It is known that $\chi^{-1}\check{Z}(K_0\cup L)$ is presented by 
\begin{align*}
\chi^{-1}\check{Z}(K_0\cup L)=\exp\Big(\df{1}{2}\sum_{x_i,x_j\in X}\mbox{$\begin{array}{c}
   \includegraphics[scale=0.2]{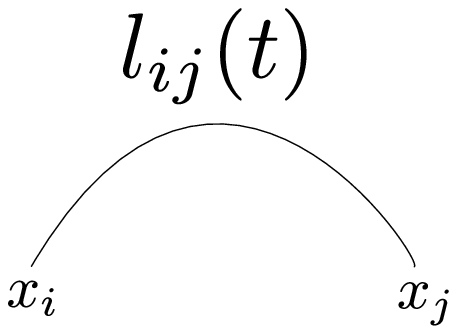}
 \end{array}$}\Big)\cup P\big(\chi^{-1}\check{Z}(K_0\cup L)\big),
\end{align*} 
where $\big(l_{ij}(t)\big)$ is an equivariant linking matrix of $L\subset S^3\backslash K_0$ which satisfies that $l_{ji}(t)=l_{ij}(t^{-1})$. Also, $P\big(\chi^{-1}\check{Z}(K_0\cup L)\big)$ is a sum of diagrams which have at least one trivalent vertex on each component, where we do not count any trivalent vertex connected to a univalent vertex labeled by $h$. For details about an equivariant linking matrix, see for example \cite{Ga0}. Then, for $n\geq 2$, the $n$-loop Kontsevich invariant of $K$ is presented as follows,
\begin{align*}
Z^{(n)}(K)=\iota_n\left(\log\left(\df{\langle\!\langle\chi^{-1}\check{Z}(K_0\cup L)\rangle\!\rangle}{\langle\!\langle\chi^{-1}\check{Z}(U_+)\rangle\!\rangle^{\sigma_+}\langle\!\langle\chi^{-1}\check{Z}(U_-)\rangle\!\rangle^{\sigma_-}}\right)\right),
\end{align*}
where $U_{\pm}$ denotes the unknot with $\pm1$ framing, and $\sigma_+$ and $\sigma_-$ are the number of the positive and negative eigenvalues of the linking matrix of $L$. 
The operation \lq\lq{$\langle\!\langle\quad\rangle\!\rangle$}" which is known as the \textit{Aarhus integral} is defined by
\begin{align*}
\langle\!\langle\chi^{-1}\check{Z}(K_0\cup L)\rangle\!\rangle=\Big\langle\exp\Big(-\df{1}{2}\sum_{x_i,x_j\in X}\mbox{$\begin{array}{c}
   \includegraphics[scale=0.2]{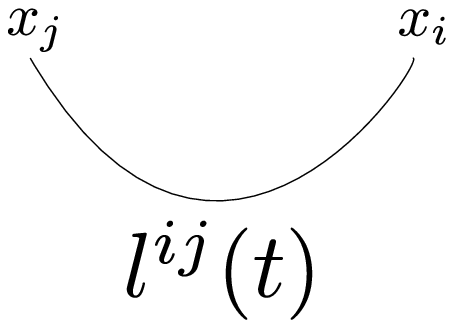}
 \end{array}$}
\Big),P\big(\chi^{-1}\check{Z}(K_0\cup L)\big)\Big\rangle,
\end{align*}
where $\big(l^{ij}(t)\big)=\big(l_{ij}(t)\big)^{-1}$, and $\langle\quad,\quad\rangle$ is defined by
\begin{align*}
\langle C_1,C_2\rangle =
\left(
\begin{tabular}{l}
\text{sum of all ways gluing the $x$-marked legs of $C_1$}\\ 
\text{to the $x$-marked legs of $C_2$ for all $x\in X$}
\end{tabular}
\right).
\end{align*}
For details, see \cite{BGRT1}. The values $\langle\!\langle\chi^{-1}\check{Z}(U_\pm)\rangle\!\rangle$ (\cite[equation (21)]{BaLa1}) are also defined in a similar way, and it is known that
\begin{align*}
\langle\!\langle\chi^{-1}\check{Z}(U_\pm)\rangle\!\rangle=\langle\chi^{-1}\nu,\chi^{-1}\nu\rangle^{-1}\exp\left(\mp\df{1}{16}{\mbox{$\begin{array}{c}
   \includegraphics[scale=0.2]{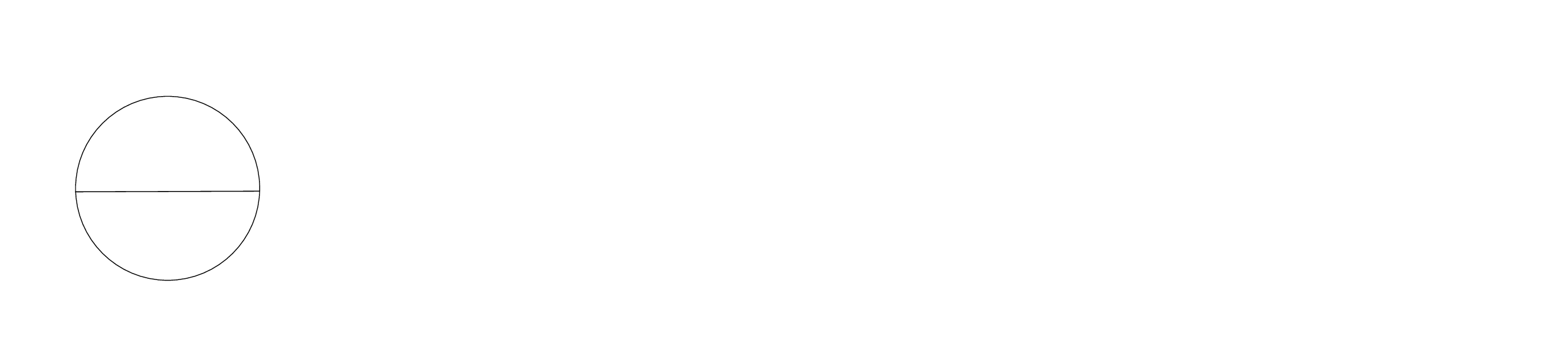}
   \end{array}$}}\right).
\end{align*}
By this procedure, we can present $\log\big(\chi^{-1}Z(K)\big)$ as
\begin{align*}
\log\big(\chi^{-1}Z(K)\big)&={\mbox{$\begin{array}{c}
   \includegraphics[scale=0.3]{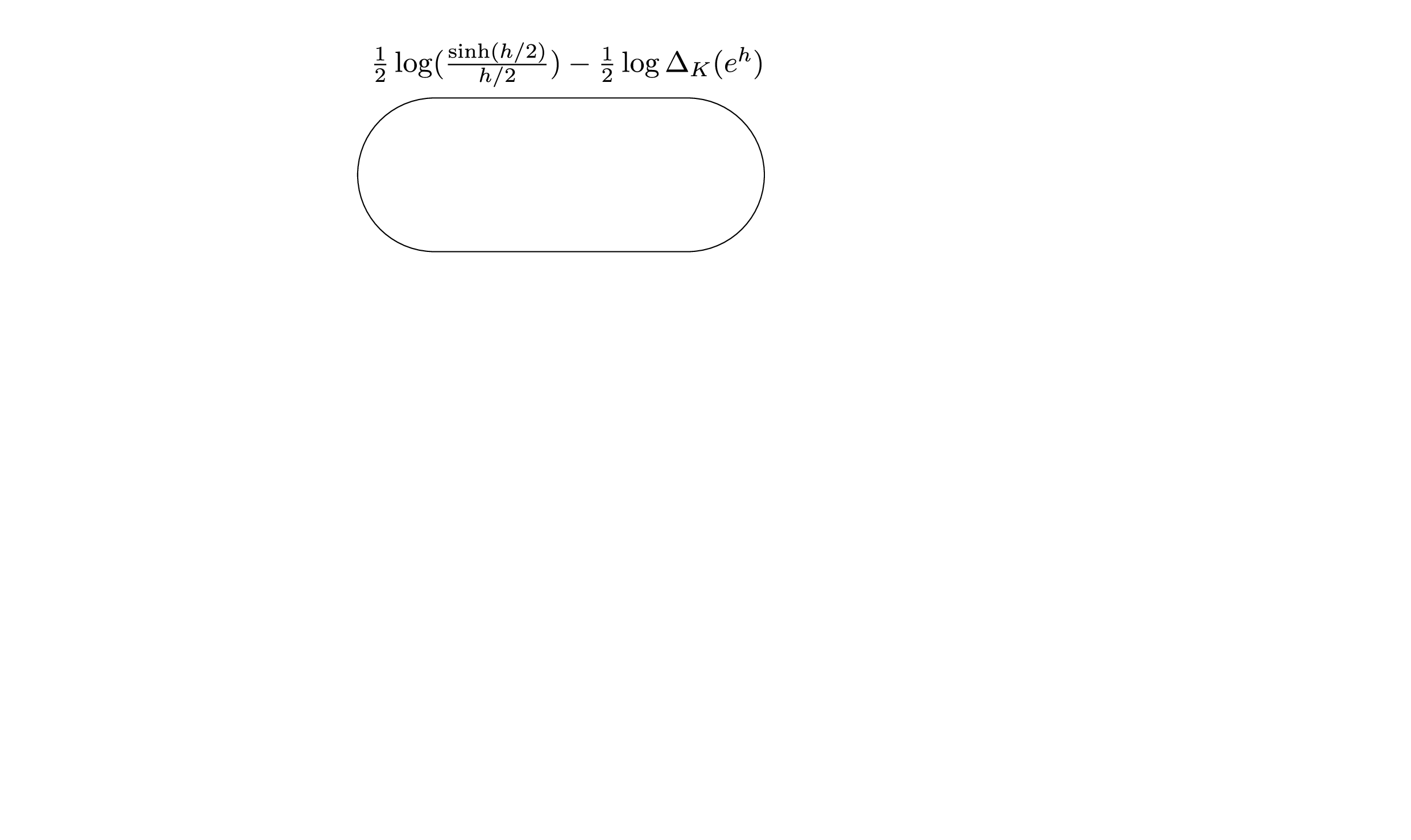}
   \end{array}$}}+\sum_{i}^{\text{finite}}{\mbox{$\begin{array}{c}
   \includegraphics[scale=0.3]{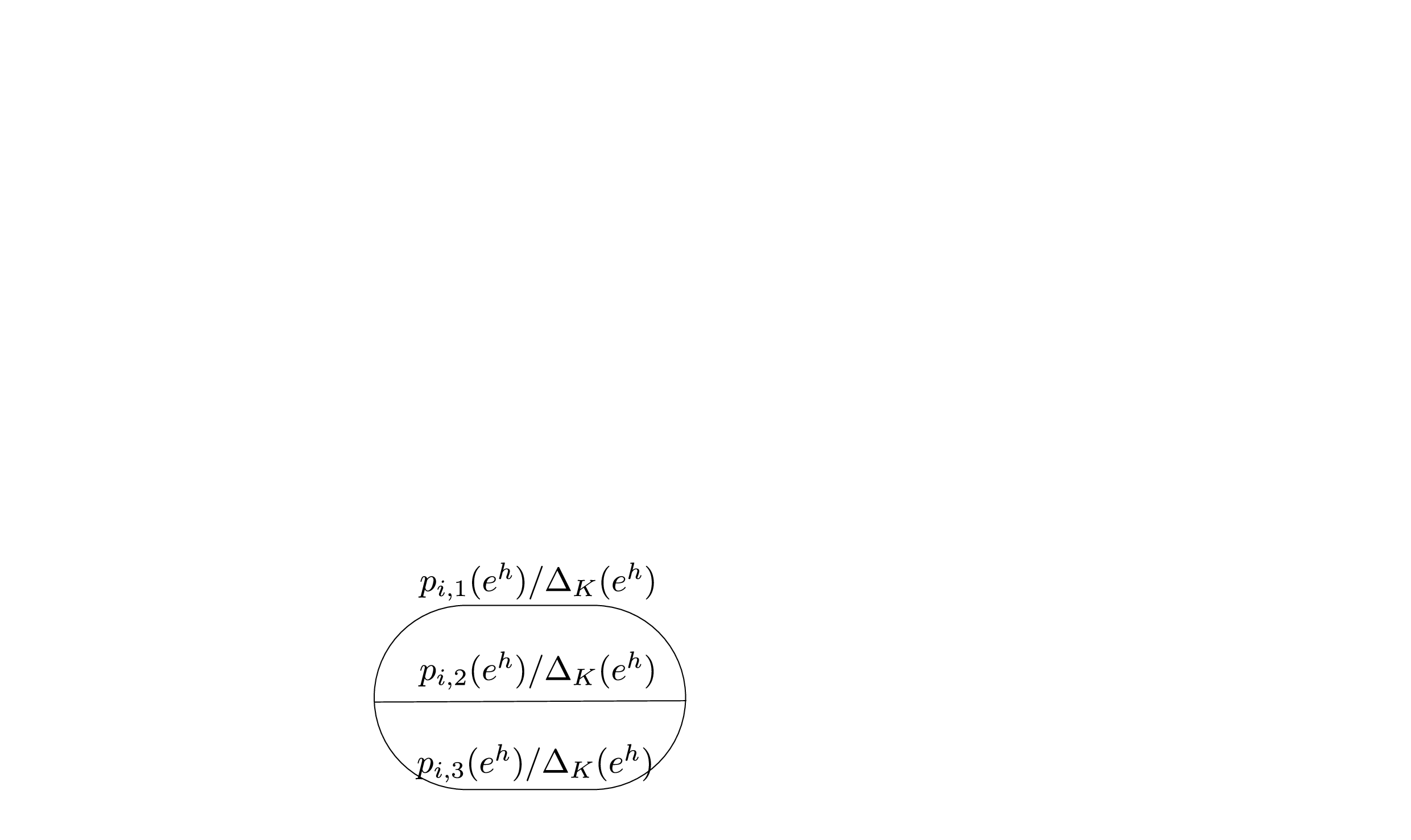}
   \end{array}$}}\\*
&+\text{(terms of ($>2$)-loop parts)},
\end{align*}
where $\Delta_K(t)$ denotes the Alexander polynomial of $K$, and by the notation (\ref{label}) we can regard the right hand side as an element in $\widetilde{\mathcal{B}_{\text{conn}}}$. Namely, $Z^{(n)}(K)$ is presented by a finite sum of $n$-loop connected open Jacobi diagram on $\emptyset$ such that one side of each edge of the diagram is labeled by a form $(\text{polynomial in $e^{\pm h}$})/\Delta_K(e^h)$. This is called the {\it loop expansion of the Kontsevich invariant of knots}.

\subsection{A review of clasper surgeries}

In this section, we briefly review clasper surgeries. For details, see \cite{Ha,Oh2}.\par
A {\it clasper} is an embedded trivalent graph in a knot complement defined by
\begin{align*}
{\mbox{$\begin{array}{c}
   \includegraphics[scale=0.15]{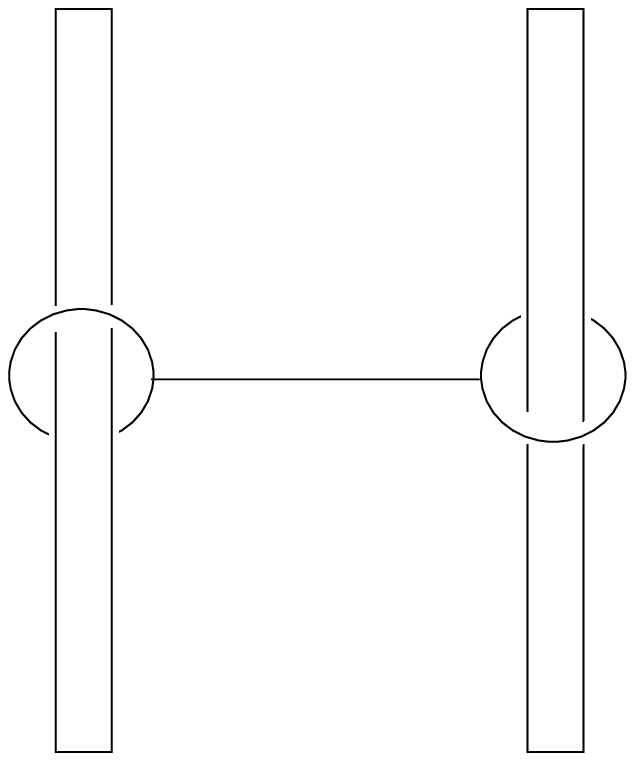}
   \end{array}$}}
&={\mbox{$\begin{array}{c}
   \includegraphics[scale=0.15]{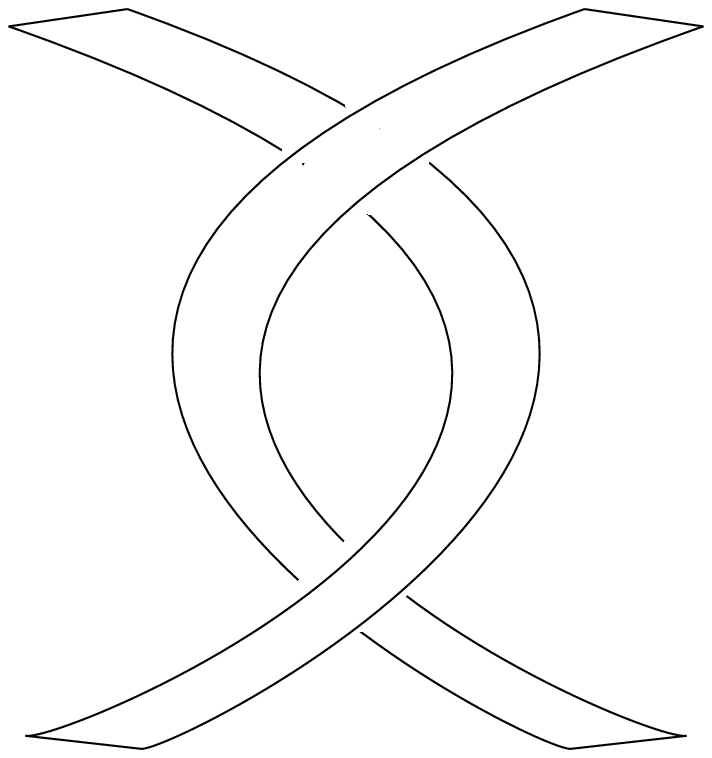}
   \end{array}$}}
\left(=\text{the result obtained by surgery along Hopf link in }{\mbox{$\begin{array}{c}
   \includegraphics[scale=0.15]{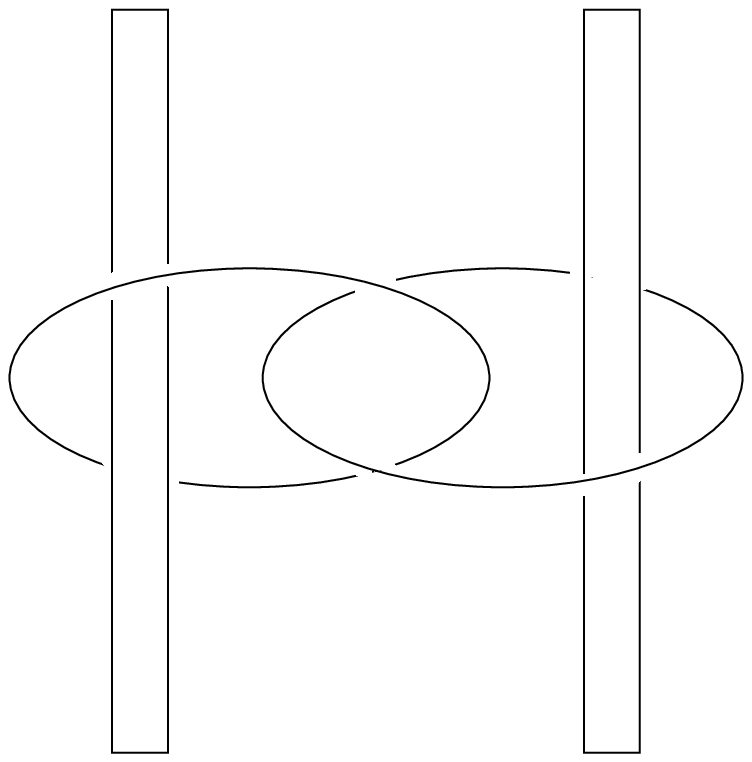}
   \end{array}$}}\right)
,\\
{\mbox{$\begin{array}{c}
   \includegraphics[scale=0.15]{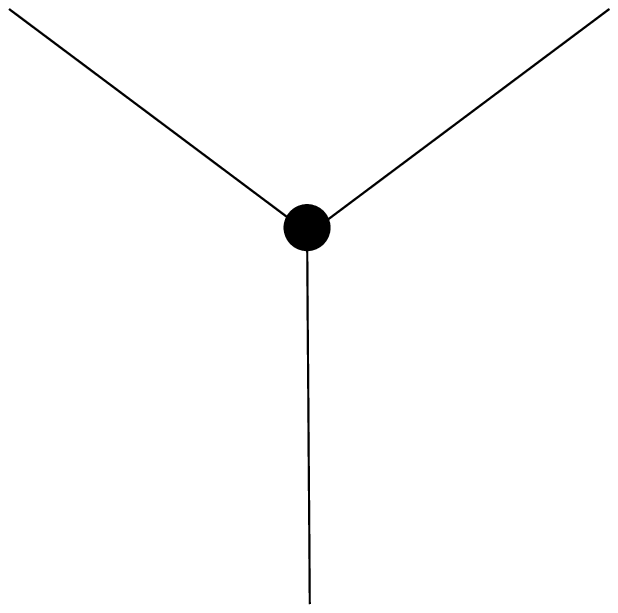}
   \end{array}$}}
&={\mbox{$\begin{array}{c}
   \includegraphics[scale=0.15]{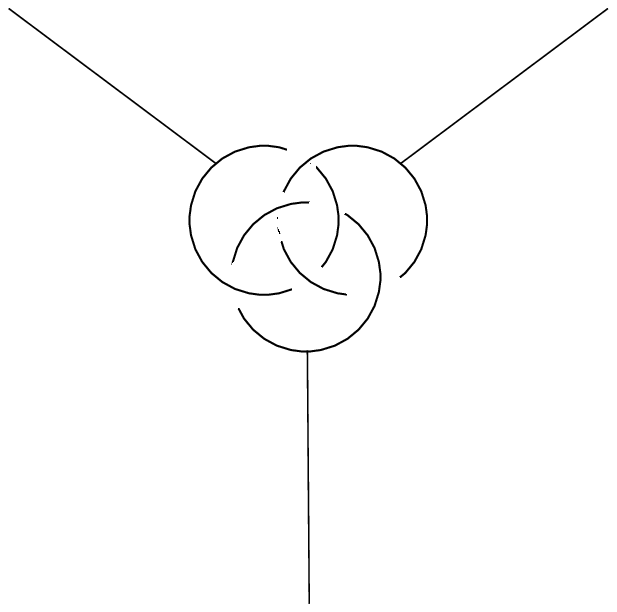}
   \end{array}$}}.
\end{align*}
We call each loop at each end of a clasper a {\it leaf}. For example, see the knots $K_1$ and $K_2$ in Proof of Theorem \ref{t21d}. A clasper surgery formula is known as follows, see for example \cite{Oh4,Oh2} for its proof,
\begin{align*}
Z\left({\mbox{$\begin{array}{c}
   \includegraphics[scale=0.3]{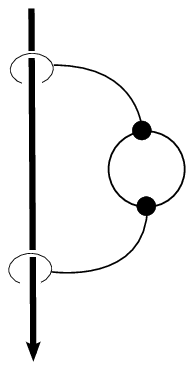}
   \end{array}$}}\right)
-Z\left({\mbox{$\begin{array}{c}
   \includegraphics[scale=0.3]{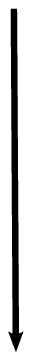}
   \end{array}$}}\right)
={\mbox{$\begin{array}{c}
   \includegraphics[scale=0.3]{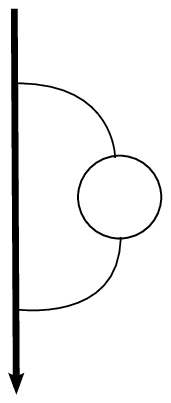}
   \end{array}$}}+(\text{terms with higher loop or higher degree}).
\end{align*}

\section{Proofs of the theorems}
\label{sec5}

In this section, we prove theorems in Section \ref{sec3}. At first, we do some preparations for their proofs. \par
Let $K$ be a genus $\leq g$ knot and $\Sigma$ a genus $g$ Seifert matrix of $K$. Then, there exists a $2g$-component framed tangle $T$ such that
\begin{align*}
T&={\mbox{$\begin{array}{c}
   \includegraphics[scale=0.3]{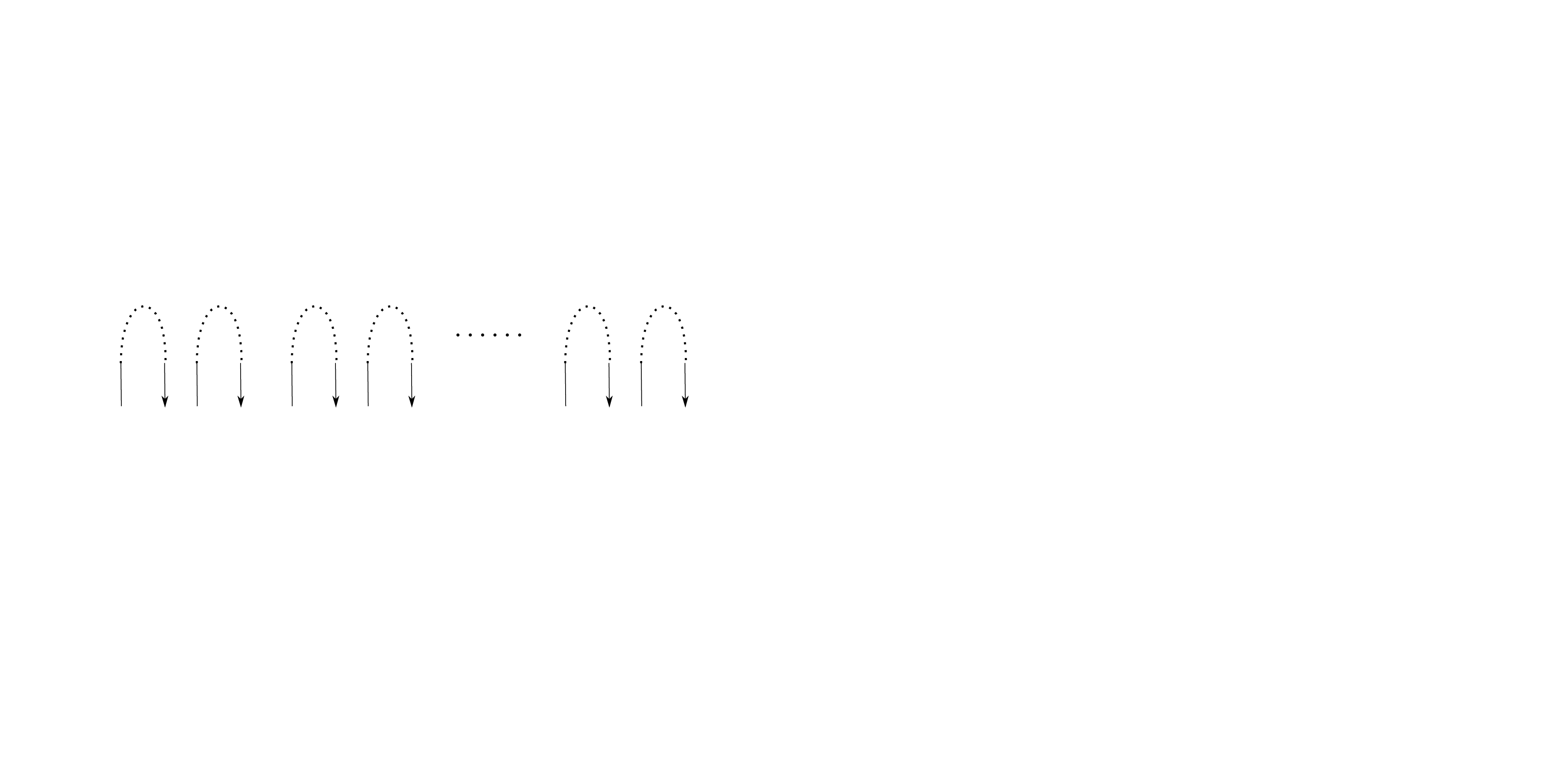}
   \end{array}$}},\\
\Sigma&={\mbox{$\begin{array}{c}
   \includegraphics[scale=0.2]{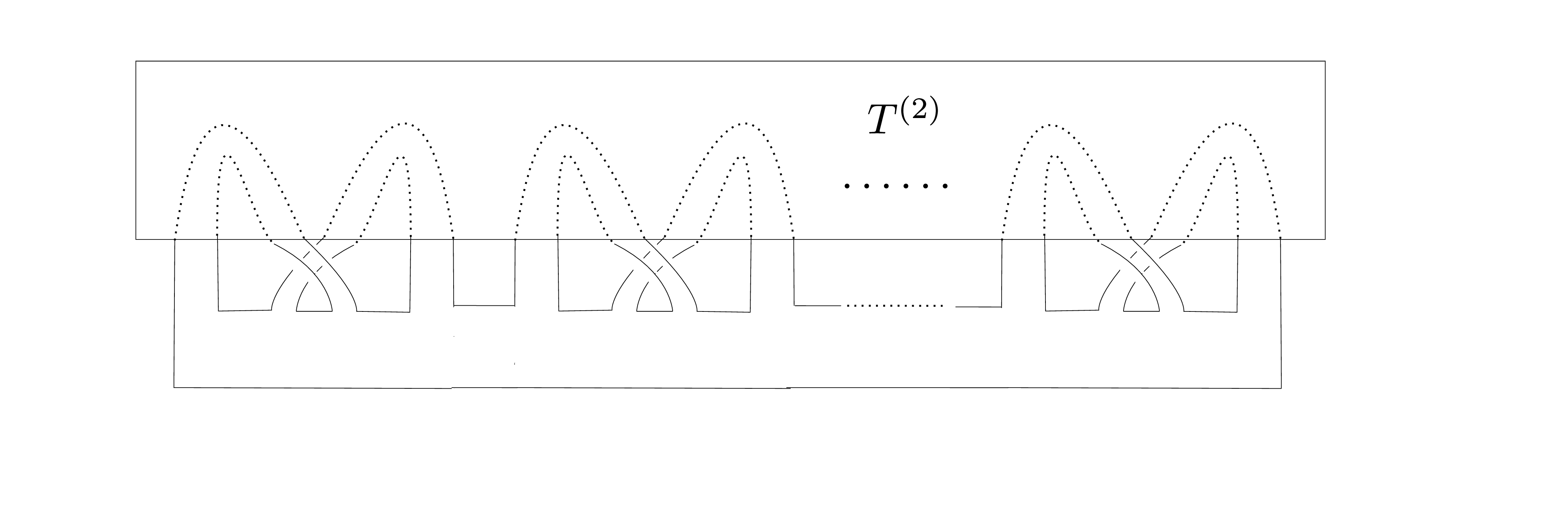}
   \end{array}$}},
\end{align*} 
where dotted lines in the picture of $T$ imply strands knotted or linked in some fashion, and $T^{(2)}$ denotes the double of $T$. We call $T$ a {\it representing tangle} of $K$, noting that a representing tangle is not unique for $K$. Then, we can obtain the following surgery presentation of $K$,
\begin{align*}
K_0\cup L={\mbox{$\begin{array}{c}
   \includegraphics[scale=0.2]{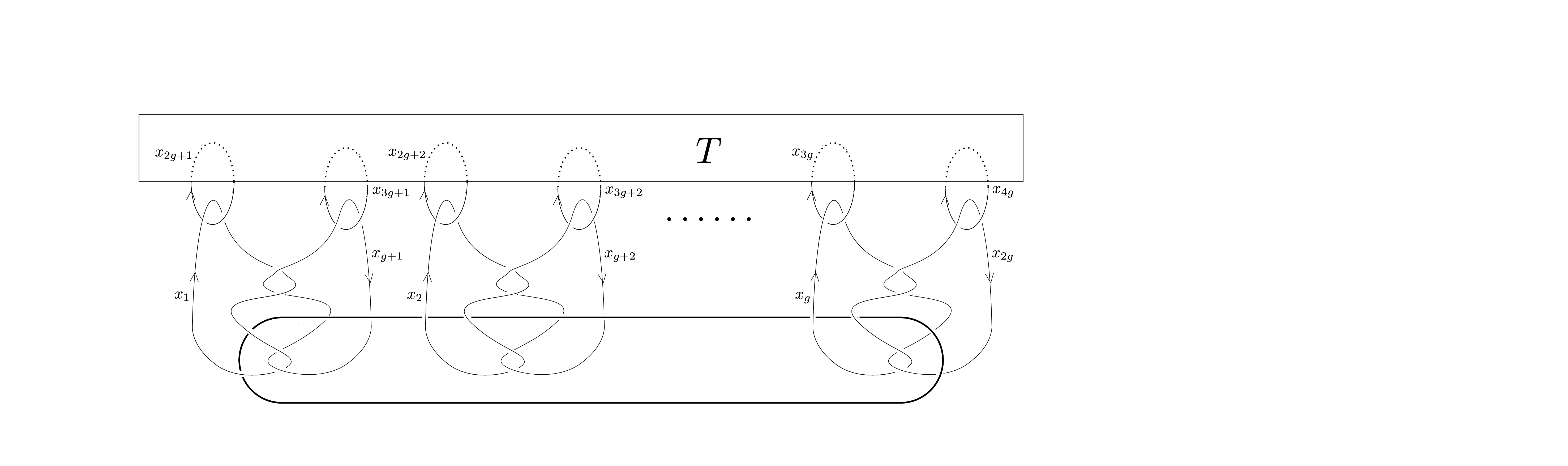}
   \end{array}$}},
\end{align*}
where $K_0$ is depicted by the thick line, $L$ is depicted by the thin lines, and $X=\{x_1,x_2,\cdots,x_{4g}\}$ is the set correspond to $L$ (see Section \ref{sec4}). Hence, its equivariant linking matrix is presented by
\begin{align*}
\big(l_{ij}(t)\big)=
\begin{pmatrix}
O&(t^{-1}-1)I&I&O\\
(t-1)I&O&O&I\\
I&O&U&W\\
O&I&W^{T}&V\\
\end{pmatrix},
\end{align*}
where $O$ denotes the zero matrix of size $g$, $I$ denotes the unit matrix of size $g$, and 
$\begin{pmatrix}
U&W\\
W^{T}&V\\
\end{pmatrix}$
is the linking matrix of $T$.\par
Let $A_t(\bigsqcup_X\downarrow)$ be the subspace of $\widetilde{\mathcal{A}(\ast_{h}\sqcup\bigsqcup_X\downarrow)}$ generated by equivalent classes of Jacobi diagrams on $\bigsqcup_X\downarrow$ whose edges are labeled by an element in $\{\emptyset,t,t^{-1}\}$, where all univalent vertices labeled by $h$ are obtained from (\ref{tehh}). Further, we also define the subspace $A_t(\ast_X)\subset\widetilde{\mathcal{A}(\ast_{\{h\}\cup X})}$ in the same way.

\begin{lem}
\label{lem1}
For any $J\in A_t(\bigsqcup_X\downarrow)$, we have $\chi_X^{-1}(J)\in A_t(\ast_X)$.
\end{lem}

\begin{proof}
It is sufficient to prove the statement when $J$ has the single term and $X=\{\text{pt}\}$. We prove it by induction on the number of univalent vertices of $J$. If $J$ has one univalent vertex, it is obvious. Suppose that $J$ has $k$ univalent vertices. Then, by the definition of $\chi$, we get
\begin{align}
\label{xdd}
\chi^{-1}\left(\mbox{$\begin{array}{c}
   \includegraphics[scale=0.2]{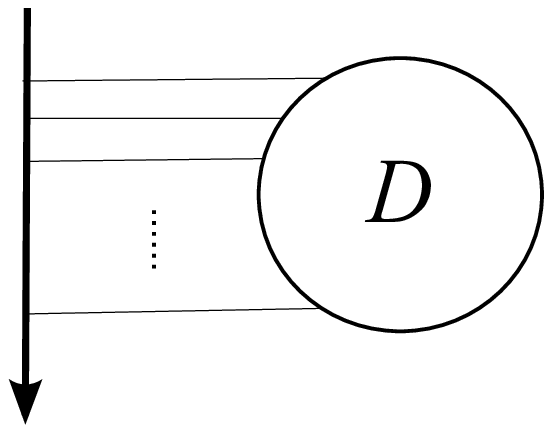}
   \end{array}$}\right)=\mbox{$\begin{array}{c}
   \includegraphics[scale=0.2]{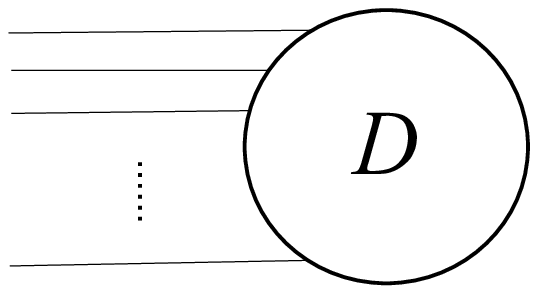}
   \end{array}$}+\df{1}{k!}\sum_{\sigma\in\mathcal{S}_k}\chi^{-1}\left(\mbox{$\begin{array}{c}
   \includegraphics[scale=0.2]{skD.ps}
   \end{array}$}-\sigma\Big(\mbox{$\begin{array}{c}
   \includegraphics[scale=0.2]{skD.ps}
   \end{array}$}\Big)\right),
\end{align}
where $J=\mbox{$\begin{array}{c}
   \includegraphics[scale=0.2]{skD.ps}
   \end{array}$}$, and $\sigma$ acts on $J$ by permutation of univalent vertices of $J$. However, the term $\mbox{$\begin{array}{c}
   \includegraphics[scale=0.2]{skD.ps}
   \end{array}$}-\sigma\Big(\mbox{$\begin{array}{c}
   \includegraphics[scale=0.2]{skD.ps}
   \end{array}$}\Big)$ can be presented by a finite sum of Jacobi diagram with $k-1$ univalent vertices by performing the STU relation, and this relation does not change the labeling $t^{\pm 1}$ at all. Thus, this term belongs to $A_t(\bigsqcup_X\downarrow)$, and by induction hypothesis, the right hand side of (\ref{xdd}) belongs to $A_t(\ast_X)$.
\end{proof}

\begin{cor}
\label{cor2}
Let $K$ be a genus $\leq g$ knot, and let $K_0\cup L$ be the surgery presentation of $K$ as above. Then, we have $\chi^{-1}\check{Z}(K_0\cup L)\in A_t(\ast_X)$.
\end{cor}

\begin{proof}
At first, by using the equation (\ref{xzt}), the labeling $t^{\pm 1}$ of $\chi_{h}^{-1}\check{Z}(K_0\cup L)$ appears only on the skeleton $\bigsqcup_X\downarrow$. However, since the linking number of $K_0$ and each component of $L$ is equal to 0, all the labeling $t^{\pm 1}$ on the skeleton can be moved on edges by the second relation of (\ref{trel}). Hence, by the straightforward calculation, the value $\chi_{h}^{-1}\check{Z}(K_0\cup L)$ can be deformed into an element in $A_t(\bigsqcup_X\downarrow)$. Thus, by Lemma \ref{lem1}, it holds that $\chi^{-1}\check{Z}(K_0\cup L)=\chi_X^{-1}\chi_h^{-1}\check{Z}(K_0\cup L)\in A_t(\ast_X)$.
\end{proof}

For positive integers $n,m$, and $\Delta(t)\in\mathcal{Z}$, let $E_0(n,m,\Delta(t))$ be the $\mathbb{Q}$-vector space spanned by equivalent classes of $n$-loop connected Jacobi diagrams on $\emptyset$ such that one side of each edge of the diagram is labeled by one of the elements in $\{\emptyset\}\cup\{t^k/\Delta(t)\mid k=-m,\cdots,-1,1,\cdots,m\}$. By (\ref{tehh}), we can regard this space as subspaces of $\widetilde{\mathcal{B}_{\text{conn}}}$, and this is finite dimensional subspaces.

\begin{proof}[Proof of Theorem \ref{thm1}]
Let $K$ be a knot in $\mathcal{K}_{\leq g}^{\Delta(t)}$, and we consider the above surgery presentation $K_0\cup L$. By Corollary \ref{cor2},
\begin{align*}
\chi^{-1}\check{Z}(K_0\cup L)=\exp\Big(\df{1}{2}\sum_{x_i,x_j\in X}\mbox{$\begin{array}{c}
   \includegraphics[scale=0.2]{st.ps}
 \end{array}$}\Big)\cup P\big(\chi^{-1}\check{Z}(K_0\cup L)\big)\in A_t(\ast_X).
\end{align*} 
Recall that its Aarhus integral is presented by
\begin{align}
\label{aahus}
\langle\!\langle\chi^{-1}\check{Z}(K_0\cup L)\rangle\!\rangle=\Big\langle\exp\Big(-\df{1}{2}\sum_{x_i,x_j\in X}\mbox{$\begin{array}{c}
   \includegraphics[scale=0.2]{stt.ps}
 \end{array}$}
\Big),P\big(\chi^{-1}\check{Z}(K_0\cup L)\big)\Big\rangle.
\end{align}
Here,
\begin{align*}
\big(l^{ij}(t)\big)=\big(l_{ij}(t)\big)^{-1}=\begin{pmatrix}
O&(t^{-1}-1)I&I&O\\
(t-1)I&O&O&I\\
I&O&U&W\\
O&I&W^{T}&V\\
\end{pmatrix}^{-1}
=\df{1}{\Delta(t)}\cdot\big(q^{ij}(t)\big),
\end{align*}
where $q^{ij}(t)\in\mathbb{Q}[t^{\pm 1}]$ is of the form $q^{ij}(t)=\displaystyle\sum_{k=-g}^{g}r_kt^k$, $r_k\in\mathbb{Q}$. Thus, since $P\big(\chi^{-1}\check{Z}(K_0\cup L)\in A_t(\ast_X)$, the logarithm of the right hand side of (\ref{aahus}) belongs to $\displaystyle\bigoplus_{n\geq 2}E_0(n,g+2,\Delta(t))$. Hence, the $n$-loop part $Z^{(n)}(K)=\iota_n\left(\log\left(\df{\langle\!\langle\chi^{-1}\check{Z}(K_0\cup L)\rangle\!\rangle}{\langle\!\langle\chi^{-1}\check{Z}(U_+)\rangle\!\rangle^{\sigma_+}\langle\!\langle\chi^{-1}\check{Z}(U_-)\rangle\!\rangle^{\sigma_-}}\right)\right)$ belongs to $E_0(n,g+2,\Delta(t))$. This subspace is finite dimensional, so it conclude that $\mathcal{V}(n,g,\Delta(t))$ is finite dimensional. \par
We can choose a basis $\{\beta_1,\cdots,\beta_d\}$ of $\mathcal{V}(n,g,\Delta(t))$ with finite number of elements. Then, it is obvious that we can write
\begin{align*}
Z^{(n)}|_{\mathcal{K}_{\leq g}^{\Delta(t)}}=c_1\beta_1+\cdots+c_d\beta_d,
\end{align*}
where $c_1,\cdots,c_d:\mathcal{K}_{\leq g}^{\Delta(t)}\to\mathbb{Q}$ are (restrictions of) Vassiliev invariants.
\end{proof}

\begin{rem}
Let $E_1(n,m,\Delta(t))$ be the $\mathbb{Q}$-vector space spanned by equivalent classes of $n$-loop connected Jacobi diagrams on $\emptyset$ such that one side of each edge of the diagram is labeled by one of the elements in $\{\emptyset\}\cup\{(t^k-1)/\Delta(t)\mid k=-m,\cdots,-1,1,\cdots,m\}$. This is also a finite dimensional subspace of $\widetilde{\mathcal{B}_{\text{conn}}}$. We have $Z^{(\geq 2)}(K)|_{t\to 1}=Z^{(\geq 2)}(K)|_{h\to 0}=0$, since it is known that this is equal the ($\geq 2$)-loop part of the LMO invariant of $S^3$, whose value is 0. Thus, in fact, $Z^{(n)}(K)$ belongs to $E_1(n,g+2,\Delta(t))$.
\end{rem}

\begin{rem}
By Proof of Theorem \ref{thm1}, we can obtain $d(n,g,\Delta(t))\leq m_n\cdot(2g+5)^{3(n-1)}$, where $m_n$ denotes the number of $n$-loop connected Jacobi diagrams on $\emptyset$, since such a diagram has $3(n-1)$ edges. However, this inequality is tremendously rough and is by no means useful.
\end{rem}

\begin{proof}[Proof of Theorem \ref{t21d}]
The ``otherwise'' case is obvious since the Alexander polynomial of any genus 1 knot is given by $1+a(t+t^{-1}-2)$, where $a\in\mathbb{Z}$. Then, assume that $\Delta(t)=1+a(t+t^{-1}-2)$ for some $a\in\mathbb{Z}$, and let $K$ be a gnus 1 knot. We put $u=t+t^{-1}-2$ and $v=t-t^{-1}$, where $t=e^h$. It is known \cite[Theorem 3.1]{Oh4} that $Z^{(2)}(K)$ is presented by $Z^{(2)}(K)=p\theta_1+q\theta_2$, where $p,q\in\mathbb{Q}$, and 
\begin{align*}
\theta_1={\mbox{$\begin{array}{c}
   \includegraphics[scale=0.25]{2loop1d.ps}
   \end{array}$}}\quad, \quad\theta_2={\mbox{$\begin{array}{c}
   \includegraphics[scale=0.25]{2loop2d.ps}
   \end{array}$}}+\left(\df{4}{3}a-\df{1}{3}\right){\mbox{$\begin{array}{c}
   \includegraphics[scale=0.25]{2loop3d.ps}
   \end{array}$}}.
\end{align*}
Thus, we obtain $\mathcal{V}(2,1,\Delta(t))\subset\text{span}_{\mathbb{Q}}(\theta_1,\theta_2)$ and $d(2,1,\Delta(t))\leq 2$.\par
Assume that $x\theta_1+y\theta_2=0$. By the straightforward calculation, we have
\begin{align*}
\theta_1&={\mbox{$\begin{array}{c}
   \includegraphics[scale=0.2]{2looph1.ps}
   \end{array}$}}+\left(-2a+\df{1}{6}\right){\mbox{$\begin{array}{c}
   \includegraphics[scale=0.2]{2looph2.ps}
   \end{array}$}}+(\text{terms with higher degree}),\\
\theta_2&=-2{\mbox{$\begin{array}{c}
   \includegraphics[scale=0.2]{2looph1.ps}
   \end{array}$}}+\left(\df{28}{3}a-\df{5}{3}\right){\mbox{$\begin{array}{c}
   \includegraphics[scale=0.2]{2looph2.ps}
   \end{array}$}}+(\text{terms with higher degree}).
\end{align*}
Hence, we have
\begin{align*}
\begin{pmatrix}
1&-2\\
-2a+\df{1}{6}&\df{28}{3}a-\df{5}{3}\\
\end{pmatrix}
\begin{pmatrix}
x\\
y
\end{pmatrix}=\begin{pmatrix}
0\\
0
\end{pmatrix}
\end{align*}
However, since $a\in\mathbb{Z}$, $\begin{vmatrix}
1&-2\\
-2a+\df{1}{6}&\df{28}{3}a-\df{5}{3}\\
\end{vmatrix}=\df{16}{3}a-\df{4}{3}\neq 0$, we have $x=y=0$. Thus, $\theta_1,\theta$ are linearly independent. Consider the following three knots,
\begin{align*}
K_0={\mbox{$\begin{array}{c}
   \includegraphics[scale=0.25]{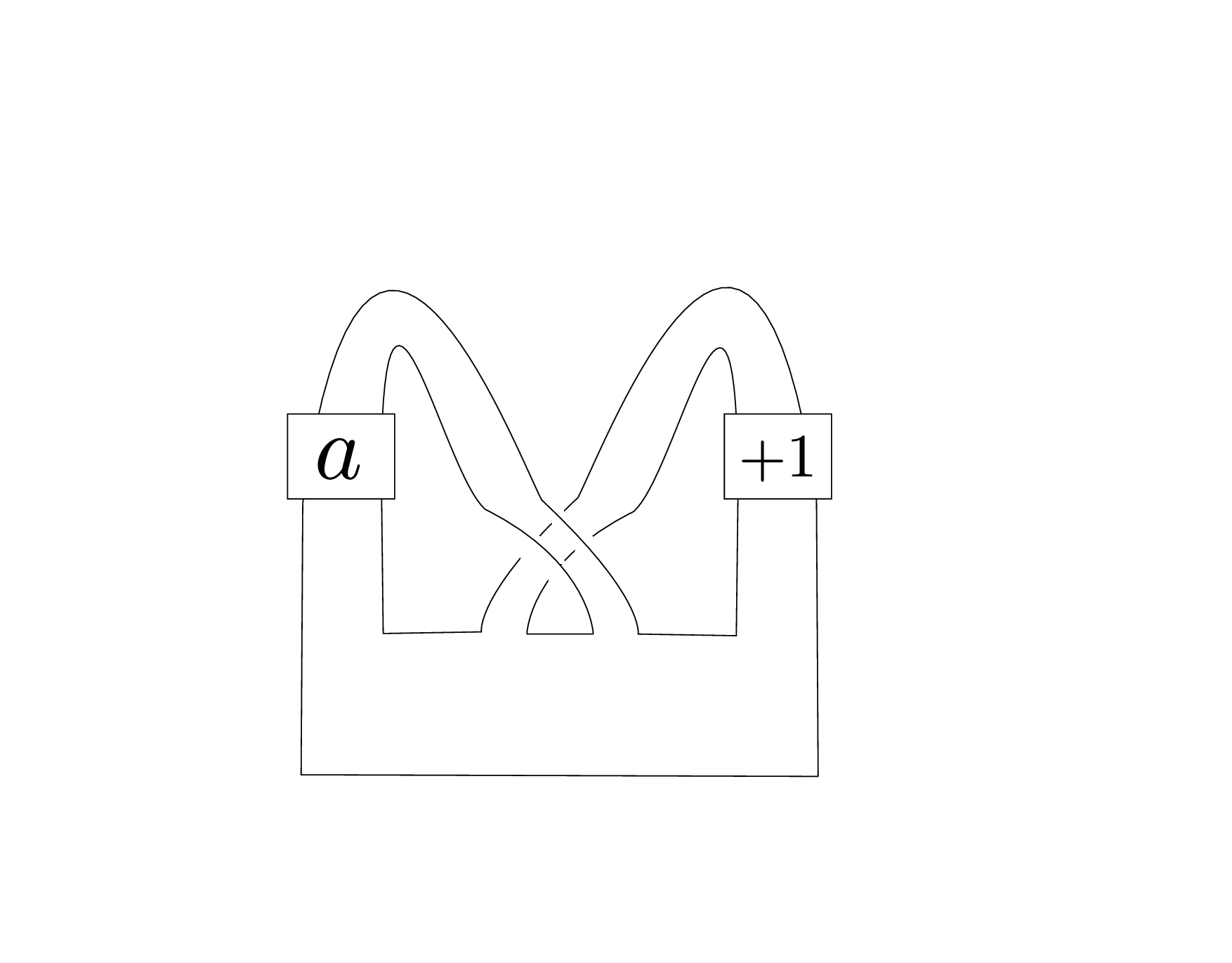}
   \end{array}$}}\quad,\quad K_1={\mbox{$\begin{array}{c}
   \includegraphics[scale=0.25]{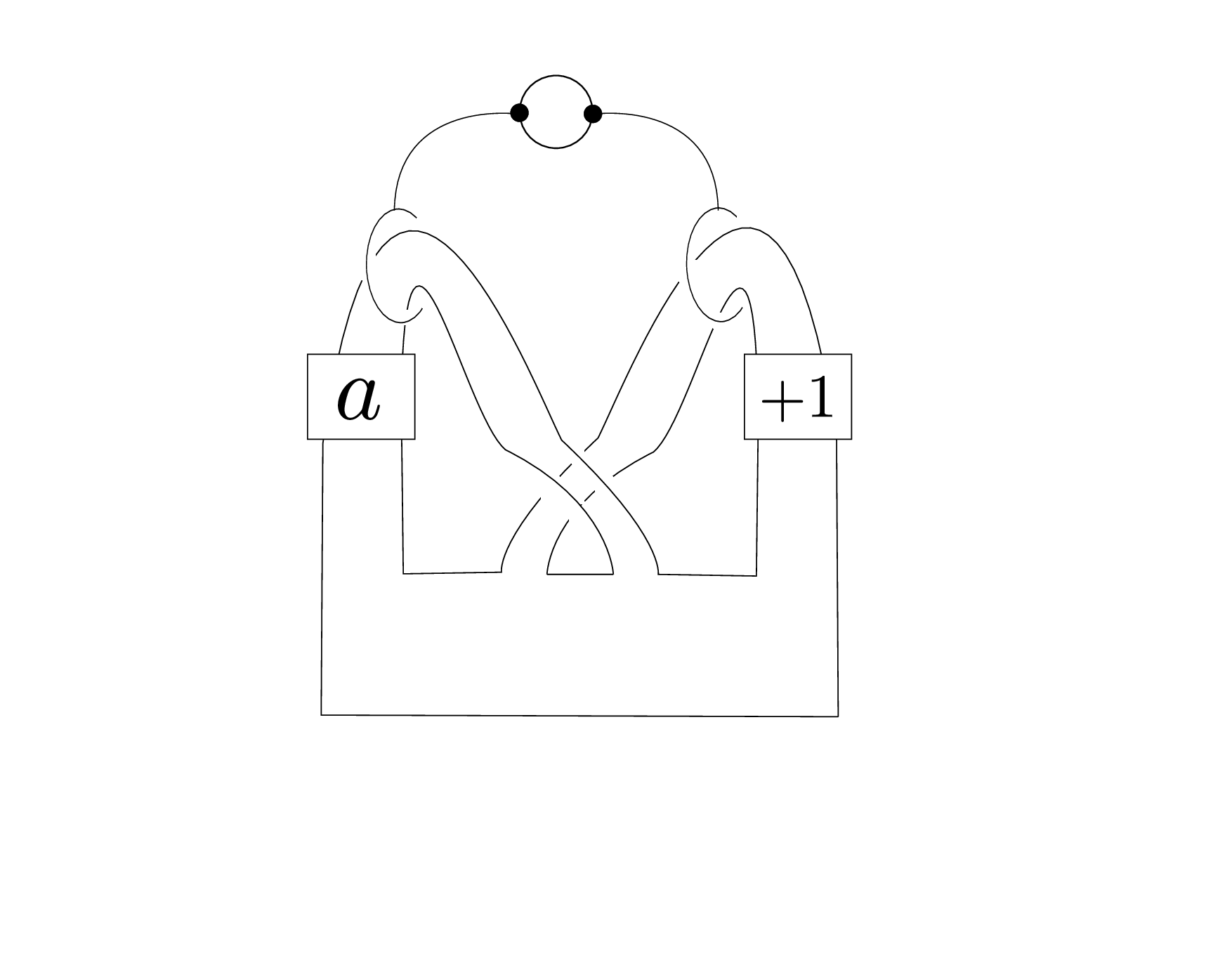}
   \end{array}$}}\quad,\quad K_2={\mbox{$\begin{array}{c}
   \includegraphics[scale=0.25]{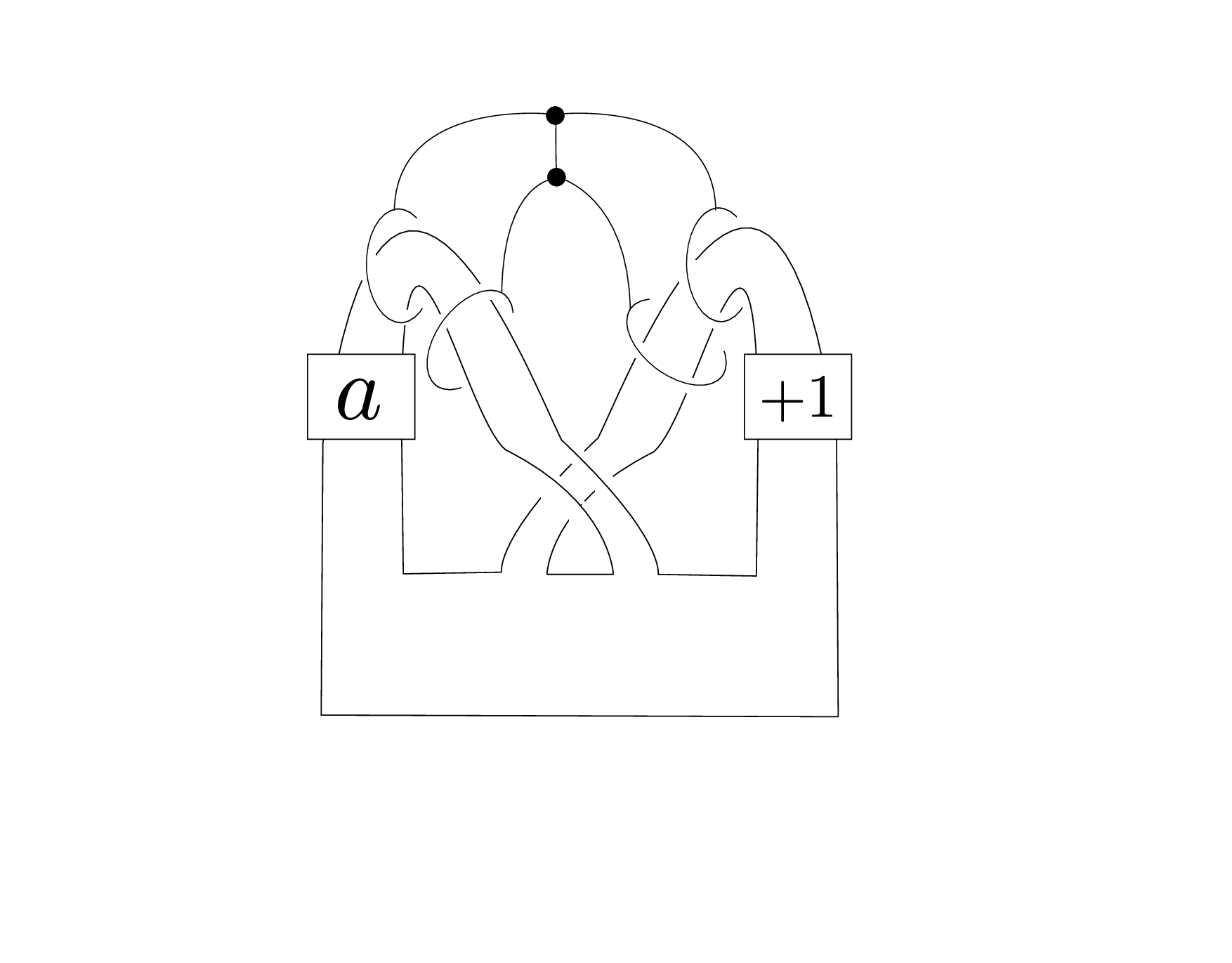}
   \end{array}$}},
\end{align*}
where ``$a$-box'' (resp. ``$+1$-box'') represents $a$-full twist (resp. $+1$-full twist), and the connected graphs in the complement of knots are claspers. Note that the Alexander polynomial of $K_i$ is equal to $\Delta(t)$. We can calculate $Z^{(2)}(K_i)$ by using Proposition 3.9, Lemma 3.11, and Lemma 3.12 of \cite{Oh4}, as follows
\begin{align*}
Z^{(2)}(K_0)&=\left(-\df{1}{16}a(a+1)\right)\theta_1+\left(\df{1}{32}a(a+1)\right)\theta_2,\\
Z^{(2)}(K_1)&=\left(-\df{1}{16}a(a+1)+\df{1}{2}\right)\theta_1+\left(\df{1}{32}a(a+1)\right)\theta_2,\\
Z^{(2)}(K_2)&=\left(-\df{1}{16}a(a+1)\right)\theta_1+\left(\df{1}{32}a(a+1)-\df{3}{4}\right)\theta_2.
\end{align*}
When $a\neq 0,-1$, we have
\begin{align*}
\begin{vmatrix}
-\df{1}{16}a(a+1)&\df{1}{32}a(a+1)\\
-\df{1}{16}a(a+1)+\df{1}{2}&\df{1}{32}a(a+1)
\end{vmatrix}=-\df{1}{64}a(a+1)\neq 0.
\end{align*}
Thus, $Z^{(2)}(K_0)$ and $Z^{(2)}(K_1)$ are linearly independent. When $a=0,-1$, we have
\begin{align*}
\begin{vmatrix}
-\df{1}{16}a(a+1)+\df{1}{2}&\df{1}{32}a(a+1)\\
-\df{1}{16}a(a+1)&\df{1}{32}a(a+1)-\df{3}{4}
\end{vmatrix}=\df{1}{16}(a-2)(a+3)\neq 0.
\end{align*}
Thus, $Z^{(2)}(K_1)$ and $Z^{(2)}(K_2)$ are linearly independent, and hence, it holds that $d(2,1,\Delta(t))\geq 2$. Therefore, we obtain $\mathcal{V}(2,1,\Delta(t))=\text{span}_{\mathbb{Q}}(\theta_1,\theta_2)$ and $d(2,1,\Delta(t))=2$.
\end{proof}

\begin{proof}[Proof of Corollary \ref{2loop}]
By Proof of Theorem \ref{t21d}, we have $Z^{(2)}(K)=p\theta_1+q\theta_2$, and
\begin{align*}
Z^{(2)}(K)=\left(p-2q\right){\mbox{$\begin{array}{c}
   \includegraphics[scale=0.2]{2looph1.ps}
   \end{array}$}}&+\left(\left(-2a+\df{1}{6}\right)p+\left(\df{28}{3}a-\df{5}{3}\right)q\right){\mbox{$\begin{array}{c}
   \includegraphics[scale=0.2]{2looph2.ps}
   \end{array}$}}\\*
&+(\text{terms with higher degree}).
\end{align*}
Hence, we have
\begin{align*}
\begin{pmatrix}
1&-2\\
-2a+\df{1}{6}&\df{28}{3}a-\df{5}{3}\\
\end{pmatrix}
\begin{pmatrix}
p\\
q
\end{pmatrix}=\begin{pmatrix}
b_1\\
b_2
\end{pmatrix},\quad\text{thus}\quad
\begin{pmatrix}
p\\
q
\end{pmatrix}=\df{3}{16a-4}
\begin{pmatrix}
\df{28}{3}a-\df{5}{3}&2\\
2a-\df{1}{6}&1\\
\end{pmatrix}
\begin{pmatrix}
b_1\\
b_2
\end{pmatrix}.
\end{align*}
Therefore, we obtain
\begin{align*}
Z^{(2)}(K)=&\left(\df{28a-5}{16a-4}b_1+\df{6}{16a-4}b_2\right){\mbox{$\begin{array}{c}
   \includegraphics[scale=0.25]{2loop1d.ps}
   \end{array}$}}\\*
&+\left(\df{12a-1}{32a-8}b_1+\df{3}{16a-4}b_2\right)\left({\mbox{$\begin{array}{c}
   \includegraphics[scale=0.25]{2loop2d.ps}
   \end{array}$}}+\left(\df{4}{3}a-\df{1}{3}\right){\mbox{$\begin{array}{c}
   \includegraphics[scale=0.25]{2loop3d.ps}
   \end{array}$}}\right)\\
=&\df{(28a-5)b_1+6b_2}{16a-4}{\mbox{$\begin{array}{c}
   \includegraphics[scale=0.25]{2loop1d.ps}
   \end{array}$}}+\df{(12a-1)b_1+6b_2}{32a-8}{\mbox{$\begin{array}{c}
   \includegraphics[scale=0.25]{2loop2d.ps}
   \end{array}$}}\\*
&+\df{(12a-1)b_1+6b_2}{24}{\mbox{$\begin{array}{c}
   \includegraphics[scale=0.25]{2loop3d.ps}
   \end{array}$}}.
\end{align*}
\end{proof}

\begin{proof}[Proof of Theorem \ref{thm6}]
It is known \cite[Theorem 4.7]{Oh4} that for any $K\in\mathcal{K}_{\leq g}^{\Delta(t)}$, the maximal degree of $t_1$ of the 2-loop polynomial of $K$ is less than or equal to $2g$, where the 2-loop polynomial is a polynomial presenting $Z^{(2)}$, see \cite{Oh4}. Thus, it follows that $Z^{(2)}(K)$ belongs to the subspace generated by the set $\{\Theta_m^n\mid m,n\in\mathbb{Z}, n\geq 1, 0\leq 2m\leq n\leq 2g\}$, where 
\begin{align*}
\Theta_m^n={\mbox{$\begin{array}{c}
   \includegraphics[scale=0.25]{2looptd.ps}
   \end{array}$}}-{\mbox{$\begin{array}{c}
   \includegraphics[scale=0.25]{2loopdd.ps}
   \end{array}$}},\quad\text{and  $t=e^h$}.
\end{align*}
Note that $\Theta_m^n$ correspond to the polynomial $T_{n,m}$ defined in Section 1 in \cite{Oh4}. Thus, since
\begin{align*}
\#\{\Theta_m^n\mid m,n\in\mathbb{Z}, n\geq 1, 0\leq 2m\leq n\leq 2g\}=\sum_{m=0}^{g}(2g-2m+1)-1=g^2+2g,
\end{align*}
we have $d(2,g,\Delta(t))\leq g^2+2g$.
\end{proof}

\section*{Appendix}

\appendix

\section{Proof of the fact that $\mathcal{V}(n,\Delta(t))$ is infinite dimensional}
\label{A}

In this section, we prove the fact that $\mathcal{V}(n,\Delta(t))$ is infinite dimensional.

\begin{prop}
\label{pa1}
For any integer $n\geq 2$ and $\Delta(t)\in\mathcal{Z}$, the subspace $\mathcal{V}(n,\Delta(t))$ is infinite dimensional.
\end{prop}

\begin{proof}
Assume that $\mathcal{V}(n,\Delta(t))$ is finite dimensional. In a similar way of Proof of Corollary \ref{corvv}, we can show that there exist finitely many Vassiliev invariants $v_1,\cdots,v_m$ such that for any $K_1,K_2\in\mathcal{K}^{\Delta(t)}$,
\begin{align}
\label{pae}
v_k(K_1)=v_k(K_2)\quad(\text{for all $k=1,\cdots,m$})\iff Z^{(n)}(K_1)=Z^{(n)}(K_2).
\end{align}
We put $d=\text{max}\{\text{degree of }v_k\mid k=1,\cdots,m\}$. Since it is known that $\mathcal{K}^{\Delta(t)}$ is not $\emptyset$, we take $K\in\mathcal{K}^{\Delta(t)}$. Then, we consider the following knot $K'$,
\begin{align*}
K'={\mbox{$\begin{array}{c}
   \includegraphics[scale=0.25]{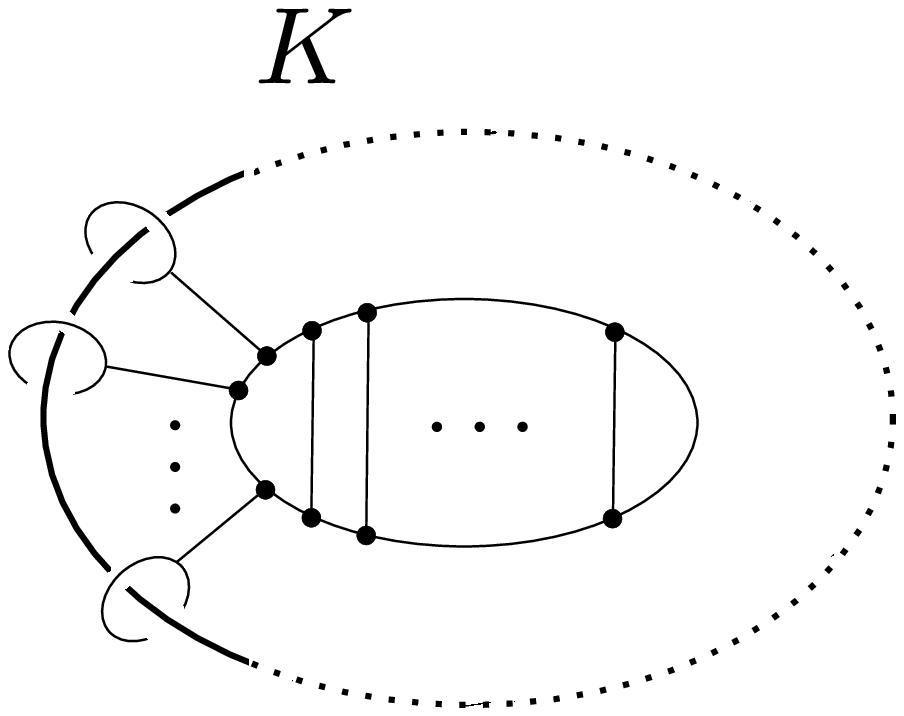}
   \end{array}$}},
\end{align*}
where the clasper in the figure is a $n$-loop clasper with $2d$ leaves. Then, by a clasper surgery formula, we have
\begin{align}
\label{paz}
Z^{(n)}(K')-Z^{(n)}(K)=\pm{\mbox{$\begin{array}{c}
   \includegraphics[scale=0.25]{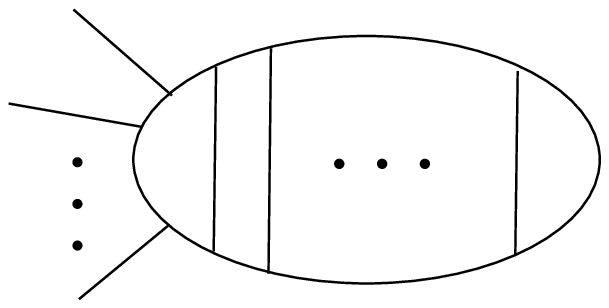}
   \end{array}$}}+(\text{terms with higher degrees}).
\end{align}
Since the first term of the right hand side has $2d$ univalent vertices, it is an open Jacobi diagram with more than degree $d$. Hence, we have $v_k(K)=v_k(K')$ for all $k=1,\cdots,m$. Thus, (\ref{pae}) implies that $Z^{(n)}(K')=Z^{(n)}(K)$. However, we can show that the $\mathfrak{sl}_2$ reduction of ${\mbox{$\begin{array}{c}
   \includegraphics[scale=0.25]{6Kd.ps}
   \end{array}$}}$ is equivalent to $(4h)^{n-1}\cdot2(2C)^dh^{2d}\neq 0$, where $C$ denotes the Casimir element of $\mathfrak{sl}_2$, see \cite{Oh3}. Thus, the diagram ${\mbox{$\begin{array}{c}
   \includegraphics[scale=0.25]{6Kd.ps}
   \end{array}$}}$ is not zero, and hence, (\ref{paz}) implies $Z^{(n)}(K')\neq Z^{(n)}(K)$. This is contradiction.
\end{proof}

\section{Proof of the fact that $\mathcal{K}_{\leq g}^{\Delta(t)}$ is an infinite set unless it is $\emptyset$}
\label{B}

In this section, we prove that the set $\mathcal{K}_{\leq g}^{\Delta(t)}$ is an infinite set unless it is $\emptyset$.

\begin{prop}
For any $\Delta(t)\in\mathcal{Z}$, there exists a integer $g_0>0$ such that 
\begin{align*}
\#\mathcal{K}_{\leq g}^{\Delta(t)}=\left\{
\begin{array}{ll}
0&(0<g<g_0)\\
\infty&(g\geq g_0)
\end{array}
\right.
\end{align*}
\end{prop}

\begin{proof}
Since it is known that $\mathcal{K}^{\Delta(t)}$ is not $\emptyset$, there exists a integer $g_0>0$ such that $\#\mathcal{K}_{\leq g}^{\Delta(t)}\neq 0$ for any $g\geq g_0$. For such $g$, we take an element $K\in\mathcal{K}_{\leq g}^{\Delta(t)}$. Then, we consider the following knot $K_n$,
\begin{align*}
K_n={\mbox{$\begin{array}{c}
   \includegraphics[scale=0.2]{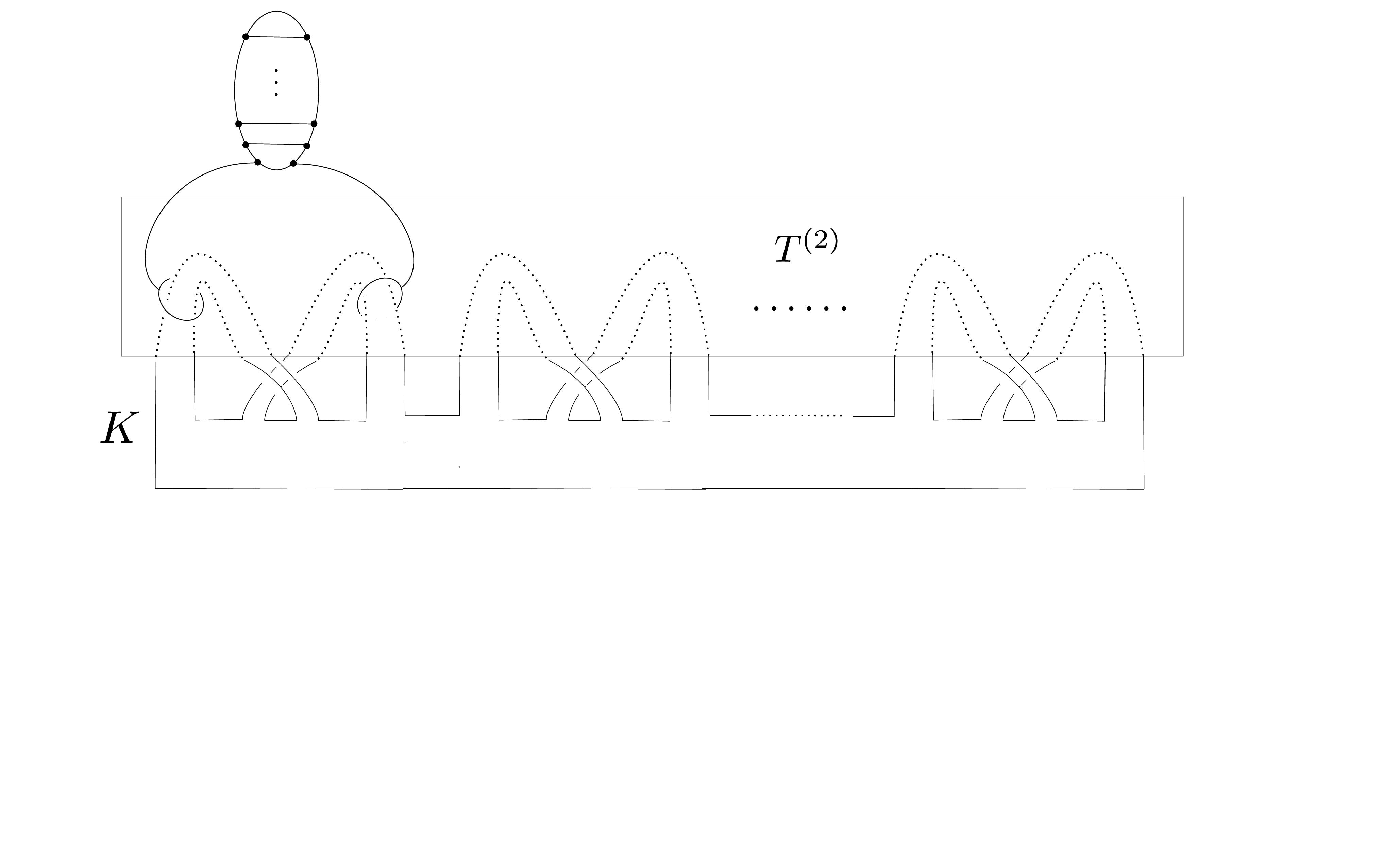}
   \end{array}$}},
\end{align*}
where the clasper in the figure is the $n$-loop clasper and $T$ is a representation tangle of $K$ with $2g$-component. We note that the Alexander polynomial of $K_n$ is also $\Delta(t)$. We obtain the following surgery presentation of $K_n$,
\begin{align*}
K_0\cup L_n={\mbox{$\begin{array}{c}
   \includegraphics[scale=0.2]{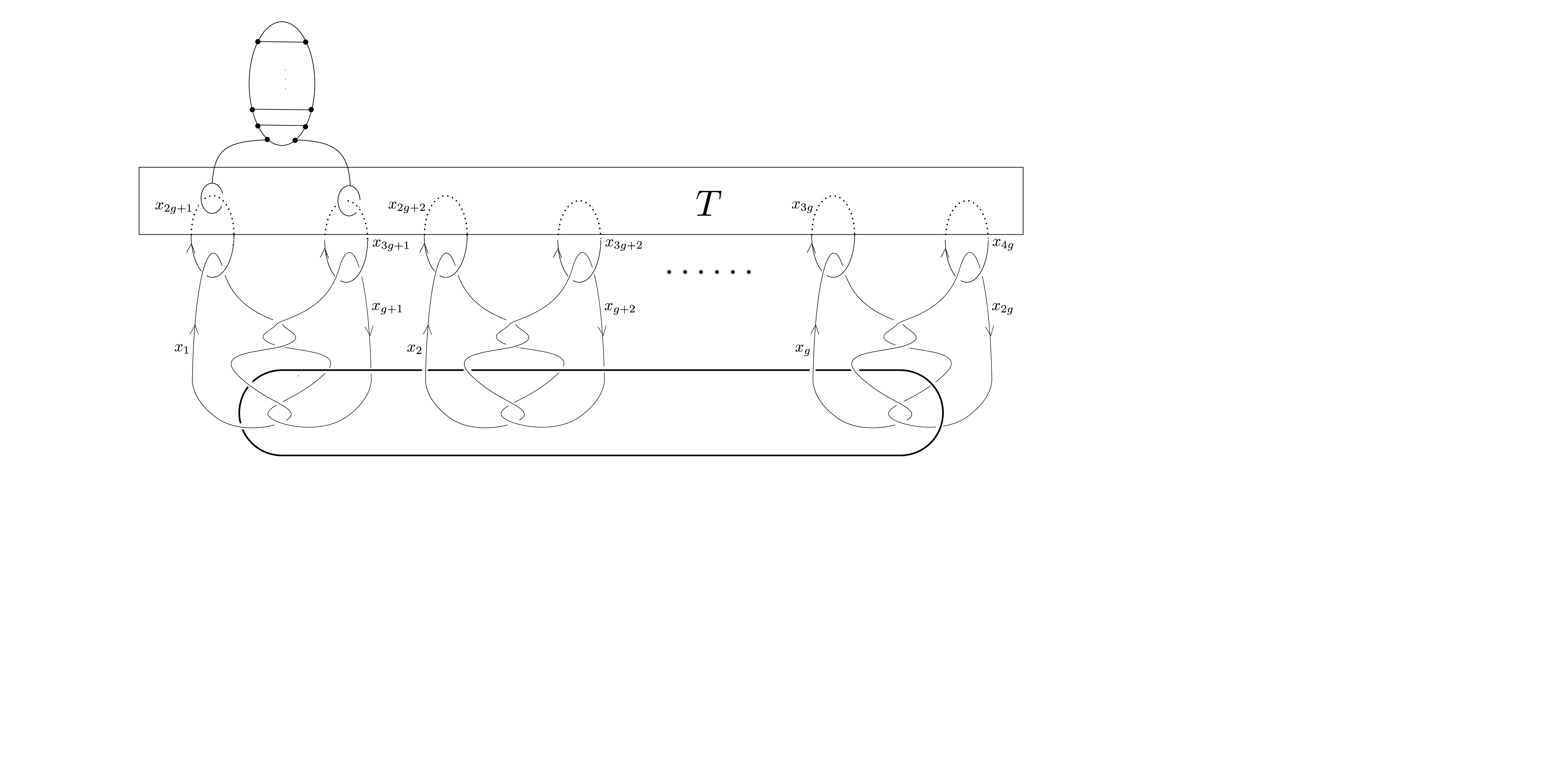}
   \end{array}$}},
\end{align*}
where $K_0$ is depicted by the thick line, $L_n$ is depicted by the thin lines. Then, its equivariant linking matrix is given by
\begin{align*}
\big(l_{ij}(t)\big)=
\begin{pmatrix}
O&(t^{-1}-1)I&I&O\\
(t-1)I&O&O&I\\
I&O&U&W\\
O&I&W^{T}&V\\
\end{pmatrix},
\end{align*}
where $O$ denotes the zero matrix of size $g$, $I$ denotes the unit matrix of size $g$, and 
$\begin{pmatrix}
U&W\\
W^{T}&V\\
\end{pmatrix}$
is the linking matrix of $T$. By a clasper surgery formula, we have
\begin{align*}
\chi^{-1}\check{Z}(K_0\cup L_n)-\chi^{-1}\check{Z}(K_0\cup L_0)=&\exp\Big(\df{1}{2}\sum_{x_i,x_j\in X}\mbox{$\begin{array}{c}
   \includegraphics[scale=0.2]{st.ps}
 \end{array}$}\Big)\cup\Big(\pm{\mbox{$\begin{array}{c}
   \includegraphics[scale=0.25]{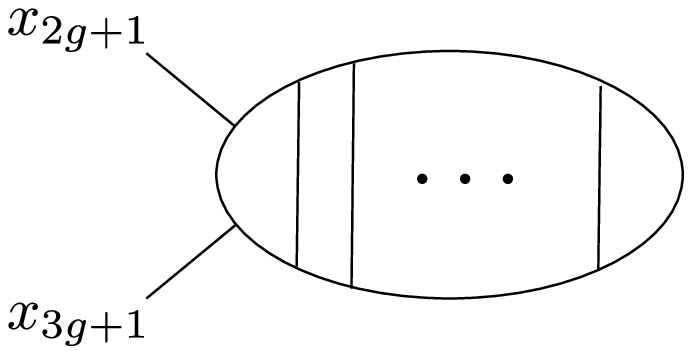}
   \end{array}$}}\\*
&+(\text{terms with more trivalent vertices})\Big),
\end{align*}
where ${\mbox{$\begin{array}{c}
   \includegraphics[scale=0.25]{6Kc.ps}
   \end{array}$}}$ is the $n$-loop Jacobi diagram corresponding to the clasper, and $K_0\cup L_0$ is the link with the clasper removed from $K_0\cup L_n$, noting that this is the surgery presentation of $K$. Thus, we have
\begin{align}
\label{cknk}
\chi^{-1}Z(K_n)-\chi^{-1}Z(K)=&\Big\langle\exp\Big(-\df{1}{2}\sum_{x_i,x_j\in X}\mbox{$\begin{array}{c}
   \includegraphics[scale=0.2]{stt.ps}
 \end{array}$}
\Big),\pm{\mbox{$\begin{array}{c}
   \includegraphics[scale=0.25]{6Kc.ps}
   \end{array}$}}\Big\rangle\nonumber\\*
&+(\text{terms of higher loops}),
\end{align}
where we denote $\big(l^{ij}(t)\big)=\big(l_{ij}(t)\big)^{-1}$. Since it holds that $\displaystyle\sum_{j=1}^{4g}l_{g+1,j}(t)l^{j,2g+1}(t)=0$, we get 
\begin{align}
\label{lgg}
l^{3g+1,2g+1}(t)=-(t-1)l^{1,2g+1}(t).
\end{align} 
However, we have
\begin{align*}
l^{1,2g+1}(t)=\df{1}{\left|\big(l_{ij}(t)\big)\right|}\cdot
\begin{vmatrix}
O'&(t^{-1}-1)I&I&O\\
(t-1)I'&O&O&I\\
I'&O'&U'&W'\\
O'&I&W^{T}&V\\
\end{vmatrix},
\end{align*}
where 
$\begin{vmatrix}
O'&(t^{-1}-1)I&I&O\\
(t-1)I'&O&O&I\\
I'&O'&U'&W'\\
O'&I&W^{T}&V\\
\end{vmatrix}$
is the $(2g+1,1)$-cofactor of $\big(l_{ij}(t)\big)$. Thus, by the straightforward calculation, we obtain
\begin{align*}
l^{1,2g+1}(1)&=\df{1}{\left|\big(l_{ij}(1)\big)\right|}\cdot
\begin{vmatrix}
O'&O&I&O\\
O'&O&O&I\\
I'&O'&U'&W'\\
O'&I&W^{T}&V\\
\end{vmatrix}
=1.
\end{align*}
This implies that
\begin{align*}
l^{1,2g+1}(t)=1+\sum_{\boldsymbol{i}=(i_+,i_-)\in\mathbb{Z}^2\backslash(0,0)}q_{\boldsymbol{i}}(t-1)^{i_+}(t^{-1}-1)^{i_-}, 
\end{align*}
where $\boldsymbol{i}=(i_+,i_-)$ is a multi index, and $q_{\boldsymbol{i}}\in\mathbb{Z}$. Thus, by (\ref{lgg}), we have
\begin{align*}
l^{3g+1,2g+1}(t)&=-(t-1)-\sum_{\boldsymbol{i}=(i_+,i_-)\in\mathbb{Z}^2\backslash(0,0)}q_{\boldsymbol{i}}(t-1)^{i_++1}(t^{-1}-1)^{i_-},
\end{align*}
Hence,
\begin{align*}
l^{3g+1,2g+1}(e^h)&=-h+rh^2+o(h^2),
\end{align*}
where $r\in1/2+\mathbb{Z}:=\{1/2+n\mid n\in\mathbb{Z}\}$, so in particular, $r\neq 0$. In the same way, we get
\begin{align*}
l^{2g+1,3g+1}(e^h)&=h+rh^2+o(h^2).
\end{align*}
Therefore, by (\ref{cknk}), we obtain
\begin{align*}
Z^{(k)}(K_n)&=Z^{(k)}(K)\quad,\quad 1\leq k\leq n,\\
Z^{(n+1)}(K_n)&=Z^{(n+1)}(K)\pm r{\mbox{$\begin{array}{c}
   \includegraphics[scale=0.25]{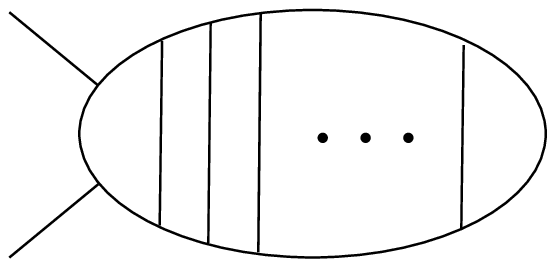}
   \end{array}$}}.
\end{align*}
By considering the $\mathfrak{sl}_2$ reduction as in the Proof of Proposition \ref{pa1}, we can show that $\pm r{\mbox{$\begin{array}{c}
   \includegraphics[scale=0.25]{6Kc2.ps}
   \end{array}$}}$ is not zero. Therefore, if $n<n'$ we have $Z^{(n+1)}(K_n)\neq Z^{(n+1)}(K_{n'})$, and hence, $K_n\neq K_{n'}$. This conclude that $\#\mathcal{K}_{\leq g}^{\Delta(t)}=\infty$.
\end{proof}

Research Institute for Mathematical Science, Kyoto University, Sakyo-ku, Kyoto, 606-8502, Japan\par

E-mail address: yamakou@kurims.kyoto-u.ac.jp

\end{document}